\documentclass[nosumlimits,twoside]{amsart}
\usepackage{amsfonts, amsmath, amssymb}
\usepackage{graphicx, subcaption}
\usepackage{float}
\usepackage{srcltx}
\usepackage[all]{xy}
\usepackage{bbm}
\usepackage{version}
\usepackage[T1]{fontenc}
\usepackage{xcolor}
\usepackage{enumerate}
\usepackage[normalem]{ulem}
\usepackage{bm}
\usepackage{latexsym}
\usepackage{stmaryrd}
\usepackage[2emode]{psfrag}
\usepackage{yhmath}
\usepackage{array}
\usepackage{dsfont}
\usepackage{mathrsfs}
\usepackage{tikz-cd}
\usepackage{comment}
\usetikzlibrary { decorations.pathmorphing, decorations.pathreplacing, decorations.shapes }

\newtheoremstyle{mytheorem}%
{10.0pt plus 2.0pt minus 2.0pt} 
{10.0pt plus 2.0pt minus 2.0pt} 
{\itshape} 
{} 
{\bfseries} 
{.} 
{ } 
{} 

\newtheoremstyle{mydefinition}%
{10.0pt plus 2.0pt minus 2.0pt} 
{10.0pt plus 2.0pt minus 2.0pt} 
{} 
{} 
{\bfseries} 
{.} 
{ } 
{} 

\newtheoremstyle{myremark}%
{10.0pt plus 2.0pt minus 2.0pt} 
{10.0pt plus 2.0pt minus 2.0pt} 
{} 
{} 
{\itshape} 
{.} 
{ } 
{} 

\theoremstyle{mytheorem}
\newtheorem{theorem}{Theorem}[section]
\newtheorem{lemma}[theorem]{Lemma}

\newtheorem{corollary}[theorem]{Corollary}
\newtheorem{proposition}[theorem]{Proposition} 

\theoremstyle{myremark}
\newtheorem{remark}[theorem]{Remark}

\theoremstyle{mydefinition}
\newtheorem{definition}[theorem]{Definition}

\newtheoremstyle{myzusatz}
 {10.0pt plus 2.0pt minus 2.0pt} 
{10.0pt plus 2.0pt minus 2.0pt} 
{\itshape} 
{} 
{\bfseries} 
{.} 
{ } 
{\thmname{#1}\thmnumber{ #2}\thmnote{ #3}}

\theoremstyle{myzusatz}

\definecolor{gray1}{gray}{0.8}
\definecolor{gray2}{gray}{0.6}
\definecolor{gray3}{gray}{0.4}
\definecolor{gray4}{gray}{0.2}

\allowdisplaybreaks
\excludeversion{invisible?}
\excludeversion{proof?}

\DeclareMathOperator{\rbiprod}{{\cdot\kern-.33em\triangleright\kern-.43em<}}
\def\lbiprod{{>\!\!\!\triangleleft\kern-.33em\cdot}}

\newcommand{\Cc}{\mathcal{C}}

\newcommand{\ot}{\otimes}

\renewcommand{\iff}{\Leftrightarrow}

\newcommand{\HBR}{\mathsf{HBR}}
\newcommand{\CHBR}{\mathsf{HBR}_{\mathrm{coc}}}

{
\left\{\begin{aligned}}
{\end{aligned}
\right.
}

\begin{document}
\title{Hopf braces and semi-abelian categories}
\author{Marino Gran, Andrea Sciandra}
\address{%
\parbox[b]{0.9\linewidth}{Institut de Recherche en Mathématique et Physique, Université Catholique de Louvain, Chemin du Cyclotron 2, B-1348 Louvain-la-Neuve, Belgium.}}
 \email{marino.gran@uclouvain.be}
\address{%
\parbox[b]{0.9\linewidth}{University of Turin, Department of Mathematics ``G.\@ Peano'', via
 Carlo Alberto 10, 10123 Torino, Italy.}}
 \email{andrea.sciandra@unito.it}
\begin{abstract}
Hopf braces have been introduced as a Hopf-theoretic generalization of skew braces. Under the assumption of cocommutativity, these algebraic structures are equivalent to matched pairs of actions on Hopf algebras, that can be used to produce solutions of the quantum Yang--Baxter equation. We prove that the category of cocommutative Hopf braces is semi-abelian and strongly protomodular. In particular, this implies that the main homological lemmas known for groups, Lie algebras and other classical algebraic structures also hold for cocommutative Hopf braces. Abelian objects are commutative and cocommutative Hopf algebras, that form an abelian Birkhoff subcategory of the category of cocommutative Hopf braces. Moreover, we show that the full subcategories of ``primitive Hopf braces'' and of ``skew braces'' form an hereditary torsion theory in the category of cocommutative Hopf braces, and that ``skew braces'' are also a Birkhoff subcategory and a localization of the latter category. Finally, we describe central extensions and commutators for cocommutative Hopf braces. 
\end{abstract}

\maketitle

\tableofcontents

\section*{Introduction}
Hopf braces were introduced by  Angiono, Galindo and Vendramin in \cite{AGV}, providing a Hopf-theoretic generalization of skew (left) braces \cite{GV}: they consist of two Hopf algebra structures on the same coalgebra, satisfying a compatibility condition that can be seen as a linearization of the one between the two group structures in a skew brace. One can describe Hopf braces in any braided monoidal category and these form a category that is equivalent to the one of bijective 1-cocycles \cite{FGRR}. Moreover, under the assumption of cocommutativity, Hopf braces are equivalent to matched pairs of actions on Hopf algebras, that can be used to produce solutions of the quantum Yang--Baxter equation \cite{AGV}. 

Skew braces are known to form a semi-abelian \cite{JMT} variety of universal algebras, and many interesting exactness properties of this category have been recently established in \cite{BFP, Bourn-Skew, FP, MLV}. 

The main goal of this paper is to investigate the categorical properties of the category $\CHBR$ of cocommutative Hopf braces, and to study some natural constructions, such as the semi-direct product decomposition and the commutator of normal subobjects, in this context. It turns out that $\CHBR$ has indeed many beautiful properties in common with some classical algebraic structures, from which also many of the known results concerning the categories of skew braces and cocommutative Hopf algebras can be easily deduced. Our work lays the foundations for many possible applications of categorical methods to further investigate Hopf braces, their (co)homology, as well as their relationship with skew braces.

After recalling the main definitions and some useful descriptions of limits \cite{Ag}, we observe that the category $\CHBR$ of cocommutative Hopf braces is protomodular (Proposition \ref{prop:protomodularity}). This is an important property that is shared with many algebraic categories, such as groups, loops, topological groups, Lie algebras, cocommutative Hopf algebras and skew braces, among others. In our context protomodularity \cite{Bourn-90} can be simply expressed as the validity of the \emph{Split Short Five Lemma}, well known in homological algebra. Moreover, we give a semi-direct product decomposition of cocommutative Hopf braces (Proposition \ref{thm:points}). Then, we prove that any morphism in the category $\CHBR$ has a factorization as a normal epimorphism (= a cokernel of some morphism) followed by a monomorphism. 
This factorisation is a key ingredient in the proof that the category $\CHBR$
is homological (Theorem \ref{Homological}), so that this latter has a pullback stable (normal epimorphism)-monomorphism factorisation system.
After giving a precise description of normal monomorphisms (=kernels) in $\CHBR$ (Proposition \ref{prop:normalmono}), we proceed to establish one of the main results of this work, namely that the category $\CHBR$ is semi-abelian (Theorem \ref{thm:semi-ab}). This implies that all the classical homological lemmas, such as the Noether Isomorphism Theorems, the ``Snake Lemma'' or the ``$3 \times 3$-Lemma'', hold true in 
$\CHBR$. Since the category of cocommutative Hopf braces is equivalent to the one of matched pairs of actions on cocommutative Hopf algebras \cite{AGV}, also the latter is a semi-abelian category.

Next, we consider abelian objects in $\CHBR$, that turn out to be the cocommutative Hopf braces such that the two algebra operations $\cdot $ and $\bullet$ are both commutative, and coincide (Proposition \ref{abelian}). The subcategory of abelian objects, which is  isomorphic to the category of commutative and cocommutative Hopf algebras, is abelian, and it is also a Birkhoff subcategory of $\CHBR$, i.e.  a full reflective subcategory that is also closed under subobjects and regular quotients.

When the base field $\Bbbk$ is algebraically closed of characteristic 0, we consider the full replete subcategory $\mathsf{SKB}$  of $\CHBR$ whose objects are the cocommutative Hopf braces having the property that the underlying Hopf algebras are group Hopf algebras. The notation $\mathsf{SKB}$ is justified by the fact that the category $\mathsf{SKB}$ is indeed equivalent to the one of skew braces. It turns out that $\mathsf{SKB}$ is the torsion-free subcategory of a hereditary torsion theory 
$ (\mathsf{PHBR}_{\mathrm{coc}},
 \mathsf{SKB})$ in $\CHBR$ (Theorem \ref{hereditary}), where the torsion subcategory $ \mathsf{PHBR}_{\mathrm{coc}}$ is 
the category  of ``primitive'' Hopf braces,  whose underlying Hopf algebras are generated as algebras by primitive elements.
The existence of this torsion theory is useful to prove that the category $\mathsf{SKB}$ is both a Birkhoff subcategory and a localization of the category $\CHBR$ (Theorem \ref{local}). 
In the last section we describe central extensions in $\CHBR$, as well as the Huq commutator \cite{Huq} of normal sub-Hopf braces (Proposition \ref{HCommutator}), linking the theory of Hopf braces to the theory of commutators in universal algebra. Indeed, by taking into account the fact that $\CHBR$ is strongly protomodular (Theorem \ref{thm:strongproto}), the Huq commutator of normal sub-Hopf braces correponds to the Smith commutator of the corresponding congruences \cite{BG}: this opens the way to some unexpected applications of universal algebraic methods in the theory of Hopf braces, that will be developed in the future.
The present work sheds a new light on the properties of the algebraic structures providing the solutions of the quantum Yang--Baxter equation. It further confirms the important observation in \cite{AGV} that several results of classical braces are still valid in the context of Hopf braces. Indeed, these latter structures organise themselves in a category that enjoys some of the strongest exactness properties that have been considered, so far, in non-abelian categorical algebra. \medskip

\noindent\textit{Notations and conventions}. We fix an arbitrary field $\Bbbk$ and all vector spaces are understood to be $\Bbbk$-vector spaces unless
otherwise specified. Moreover, by a linear map we mean a $\Bbbk$-linear map, and the unadorned tensor
product $\otimes$ stands for $\otimes_{\Bbbk}$. All linear maps whose domain is a tensor product will usually be defined on generators and understood to be extended by linearity. Algebras over $\Bbbk$ will be associative and unital and coalgebras over $\Bbbk$ will be coassociative and counital. 

The symbols
$\cdot,\bullet$ will be used for the multiplication of an algebra $A$. Equivalently, the multiplication will also be denoted by $m, m_{\cdot}$ or $m_{\bullet}:A\otimes A\to A$. The unit is denoted by $1$ or, equivalently, by $u:\Bbbk\to A$. The comultiplication and the counit of a coalgebra will be denoted by $\Delta$ and $\varepsilon$, respectively, while the antipode of a Hopf algebra will be denoted by $S$ or $T$. Subscripts will be added for clarity, whenever needed. We will adopt Sweedler’s notation for calculations involving the comultiplication, so that we'll write $\Delta(x)=x_{1}\otimes x_{2}$ (with the usual summation convention, where $x_{1}\otimes x_{2}$ stands for the finite sum $\sum_{i}{x^{i}_{1}\otimes x^{i}_{2}}$).

Given an object $X$ in a category $\Cc$, the identity morphism on $X$ will be denoted either by $\mathrm{Id}_{X}$ or $1_{X}$. If the category $\Cc$ is pointed, the zero object is denoted by $0$. Given a morphism $f:A\to B$ in $\Cc$, the kernel and the cokernel of $f$ in $\Cc$ will be denoted by $\mathsf{ker}(f):\mathsf{Ker}(f)\to A$ and $\mathsf{coker}(f):B\to \mathsf{Coker}(f)$, respectively. Given a monoidal category $\Cc$, we denote the categories of monoids and of comonoids in $\Cc$ by $\mathsf{Mon}(\Cc)$ and  $\mathsf{Comon}(\Cc)$, respectively. If $\Cc$ is a braided monoidal category, bimonoids and Hopf monoids in $\Cc$ are denoted by $\mathsf{Bimon}(\Cc)$ and $\mathsf{Hopf}(\Cc)$, respectively.

\section{Preliminaries}

In this section we recall some notions and results on Hopf braces and semi-abelian categories that will be useful in the following. We refer the reader to \cite{Sw} and \cite{BB} for more details about Hopf algebras and semi-abelian categories, respectively. \medskip

\noindent\textbf{Hopf braces}. The notion of Hopf brace was introduced in \cite{AGV}, providing a Hopf-theoretic generalization of the notion of skew brace \cite{GV}: it consists of two Hopf algebra structures on the same coalgebra, satisfying a compatibility condition that can be seen as a ``linearization'' of the one in the definition of a skew brace: 

\begin{definition}[{\cite[Definition 1.1]{AGV}}]
A \textit{Hopf brace} is a datum $(H,\cdot,\bullet,1,\Delta,\varepsilon,S,T)$ where $(H,\cdot,1,\Delta,\varepsilon,S)$ and $(H,\bullet,1,\Delta,\varepsilon,T)$ are Hopf algebras satisfying the following compatibility condition:
\begin{equation}\label{eq:HBRcomp.}
a\bullet(b\cdot c)=(a_{1}\bullet b)\cdot S(a_{2})\cdot(a_{3}\bullet c), \quad \text{for all}\ a,b,c\in H.
\end{equation}
We will often denote a Hopf brace simply by $(H,\cdot,\bullet)$ and adopt the notations $H^{\cdot}:=(H,\cdot,1,\Delta,\varepsilon,S)$ and $H^{\bullet}:=(H,\bullet,1,\Delta,\varepsilon,T)$. A morphism of Hopf braces $f:(H,\cdot,\bullet)\to(K,\cdot,\bullet)$ is a Hopf algebra map with respect to both the Hopf algebra structures, i.e. it satisfies
\[
f(a\cdot b)=f(a)\cdot f(b),\quad f(a\bullet b)=f(a)\bullet f(b),\quad f(1_{H})=1_{K}
\]
and
\[
\varepsilon(f(a))=\varepsilon(a),\quad \Delta(f(a))=f(a_{1})\otimes f(a_{2}), 
\]
for all $a,b\in H$. Any morphism of Hopf braces automatically preserves the antipodes, so that $f\circ S_{H}=S_{K}\circ f$ and $f\circ T_{H}=T_{K}\circ f$. A Hopf brace is \textit{cocommutative} if the comultiplication $\Delta$ is cocommutative. We denote the category of Hopf braces by $\HBR$ and its full subcategory of cocommutative Hopf braces by $\CHBR$.
\end{definition}

Any Hopf algebra $H^{\cdot}$ produces a Hopf brace $(H,\cdot,\cdot)$, called \textit{trivial} Hopf brace. Among the examples of (cocommutative) Hopf braces there are group algebras of braces \cite{Rump}, skew braces, and semi-direct products of cocommutative Hopf algebras, see \cite[Examples 1.4, 1.5, 1.6]{AGV}. \medskip

As it is proved in \cite[Lemma 1.8]{AGV}, the Hopf algebra $H^{\cdot}$ of a Hopf brace $(H,\cdot,\bullet)$ is a left $H^{\bullet}$-module algebra, i.e. 
an object in $\mathsf{Mon}(_{H^{\bullet}}\mathfrak{M})$, where the $H^{\bullet}$-action is defined, for all $a,b\in H$, by:
\begin{equation}\label{def:leftaction}
a\rightharpoonup b:=S(a_{1})\cdot(a_{2}\bullet b).
\end{equation}
In particular, the following equations are satisfied:
\[
a\rightharpoonup1=\varepsilon(a)1,\qquad a\rightharpoonup(b\cdot c)=(a_{1}\rightharpoonup b)\cdot(a_{2}\rightharpoonup c).
\]
Moreover, from the definition of Hopf brace, one obtains: 
\[
a\bullet b=a_{1}\cdot(a_{2}\rightharpoonup b),\qquad a\cdot b=a_{1}\bullet(T(a_{2})\rightharpoonup b),\qquad S(a)=a_{1}\rightharpoonup T(a_{2}).
\]
 In case the Hopf brace $(H,\cdot,\bullet)$ is cocommutative, something more can be said. Indeed, as it is proved in \cite[Lemma 2.2]{AGV}, $H^{\cdot}$ becomes a left $H^{\bullet}$-module coalgebra, i.e. an object in $\mathsf{Comon}(_{H^{\bullet}}\mathfrak{M})$, which means that the following equations are satisfied:
\[
   \Delta(a\rightharpoonup b)=(a_{1}\rightharpoonup b_{1})\otimes(a_{2}\rightharpoonup b_{2}),\quad \varepsilon(a\rightharpoonup b)=\varepsilon(a)\varepsilon(b),
   \]
where $\rightharpoonup$ is defined as in \eqref{def:leftaction}. Observe that, since $H^{\cdot}$ is in $\mathsf{Mon}(_{H^{\bullet}}\mathfrak{M})$ and in $\mathsf{Comon}(_{H^{\bullet}}\mathfrak{M})$ and it is a standard bialgebra, then it is automatically in $\mathsf{Bimon}(_{H^{\bullet}}\mathfrak{M})$, considering $_{H^{\bullet}}\mathfrak{M}$ with braiding given by the flip map (since $H^{\bullet}$ is cocommutative). Moreover, $S(a\rightharpoonup b)=a\rightharpoonup S(b)$ also holds true, so $H^{\cdot}$ is in $\mathsf{Hopf}(_{H^{\bullet}}\mathfrak{M})$. 

Cocommutative Hopf braces produce solutions of the quantum Yang--Baxter equation. Indeed, one can also define a right $H^{\bullet}$-action on $H^{\cdot}$ by
$$a\leftharpoonup b:=T(a_{1}\rightharpoonup b_{1})\bullet a_{2}\bullet b_{2}$$ for all $a,b\in H$,
that makes $H^{\cdot}$ an object in $\mathsf{Comon}(\mathfrak{M}_{H^{\bullet}})$ (see \cite[Lemma 2.2]{AGV}) and
\[
c:H\otimes H\to H\otimes H,\ a\otimes b\mapsto(a_{1}\rightharpoonup b_{1})\otimes(a_{2}\leftharpoonup b_{2})
\]
is a coalgebra isomorphism and a solution of the braid equation (or, equivalently, $\tau c$ is a solution of the quantum Yang--Baxter equation, where $\tau$ is the flip map), see \cite[Corollary 2.4]{AGV}. \medskip

\noindent\textbf{Semi-abelian categories}.
A challenging problem in categorical algebra is to identify the fundamental axioms 
allowing one to unify and investigate the homological properties that groups, Lie algebras, commutative algebras, crossed modules and many other group-like structures all have in common.
A remarkable axiomatic framework was proposed in \cite{JMT} by introducing the notion of \emph{semi-abelian} category.
We now briefly recall the definition and some basic properties that will be useful in this article.

First recall that a morphism $p \colon A \rightarrow I$ is called a regular epimorphism if it is the coequalizer of two morphisms in $\mathcal C$.
A finitely complete category $\mathcal{C}$ is \emph{regular} if: 
\begin{itemize}
\item any morphism $f \colon A \rightarrow B$ has a factorization
$f = ip $
\[\begin{tikzcd}
	A && B \\
	& I
	\arrow[from=1-1, to=1-3, "f"]
	\arrow[from=1-1, to=2-2, "p"']
	\arrow[from=2-2, to=1-3, "i"']
\end{tikzcd}\]
where $p$ is a regular epimorphism and $i$ is a monomorphism;
\item regular epimorphisms are pullback stable.
\end{itemize}
A regular category is \emph{exact} in the sense of Barr \cite{Barr} if any internal equivalence  relation $(R,r_1, r_2)$  in $\mathcal C$ on an object $A$ is a kernel pair, i.e. if it can be obtained as a pullback of a morphism $f \colon A \rightarrow B$ along itself:
\[\begin{tikzcd}
	R & A \\
	A & B
	\arrow[from=1-1, to=1-2,"r_2"]
	\arrow[from=1-1, to=2-1, "r_1"']
	\arrow[from=1-2, to=2-2,"f"]
	\arrow[from=2-1, to=2-2,"f"']
\end{tikzcd}\]
Any variety of universal algebras is an exact category, with internal equivalence relations being usually called congruences, and regular epimorphisms being the surjective homomorphisms. Any elementary topos is an exact category, as is more generally any pretopos, as the category of compact Hausdorff spaces, for instance.

A pointed category is a category with a zero object, denoted by $0$. A crucial notion for our work will be the one of (Bourn) protomodular category \cite{Bourn}. In a finitely complete 
pointed category $\mathcal C$, the protomodularity can be expressed as follows:
given any commutative diagram in $\mathcal C$

\begin{equation}
\begin{tikzcd}\label{SplitShort}
	\mathsf{Ker}(f) & A & B \\
	\mathsf{Ker}(f') & A' & B'
	\arrow[ from=1-1, to=1-2, "k"]
	\arrow[ from=1-1, to=2-1, "u"']
	\arrow[shift left, from=1-2, to=1-3, "f"]
	\arrow[from=1-2, to=2-2, "v"]
	\arrow[shift left, from=1-3, to=1-2, "s"]
	\arrow[from=1-3, to=2-3, "w"]
	\arrow[ from=2-1, to=2-2, "k'"]
	\arrow[shift left, from=2-2, to=2-3, "f'"]
	\arrow[shift left, from=2-3, to=2-2, "s'"]
\end{tikzcd}
\end{equation}
where $fs = \mathrm{Id}_B$, $f' s'  = \mathrm{Id}_{B'}$, $k= \mathsf{ker}(f)$, $k'= \mathsf{ker}(f')$, $u$ and $w$ are isomorphisms, then $v$ is an isomorphism.
In other words, a finitely complete pointed category is protomodular precisely when the \emph{Split Short Five Lemma} holds in $\mathcal C$. 
We are now ready to give the definition of semi-abelian category:
\begin{definition}\cite{JMT} \label{semi-abelian}
A pointed category is \emph{semi-abelian} if it is exact, protomodular, and it has binary coproducts.
\end{definition}
Plenty of interesting varieties in classical algebra are semi-abelian: groups, loops, rings (not necessarily unitary), Lie algebras, associative algebras, crossed modules, skew braces, any variety of $\Omega$-groups in the sense of Higgins \cite{Higgins}. Other examples of interest are the categories of cocommutative Hopf algebras over a field \cite{GSV}, commutative ${\mathbb C}^*$-algebras, the dual of the category of pointed sets, Heyting semi-lattices \cite{Johnstone}. Any abelian category is a semi-abelian category, of course. An important difference between semi-abelian and abelian categories comes from the fact that semi-abelian ones are not additive, in general.
Note that, among the varieties of algebras that are protomodular, but not semi-abelian, there are all the varieties whose theory contain more than one constant, such as unitary commutative rings and $\mathsf{MV}$-algebras, for instance. These categories are exact protomodular and cocomplete, but they are not pointed, since the initial and the terminal objects are not isomorphic.

A semi-abelian category $\mathcal C$ is  \emph{strongly protomodular} \cite{Bourn-00} if it has the following additional property:
in any diagram \eqref{SplitShort} in $\mathcal C$ where $u$ is a kernel, $B=B'$ and $w = \mathrm{Id}_B$, then the arrow $k'u$ is a kernel.
As shown in \cite{BG} these categories have a very useful property, namely that two equivalence relations $R$ and $S$ on an object $A$ admit a ``connector'' if and only if the corresponding normal monomorphisms $N_R$ and $N_S$ commute in the sense of Huq. From this observation it follows that, when the category is also semi-abelian so that both the (Smith-Pedicchio) commutator of equivalence relations \cite{Smith, Pedicchio} and the (Huq) commutator \cite{Huq} of normal subobjects can be defined, then
$R$ and $S$
have trivial commutator $[R,S]_{\mathsf{Smith}}$ in the sense of Smith if and only if $N_R$ and $N_S$ have trivial commutator $[N_R,N_S]_{\mathsf{Huq}}$ in the sense of Huq. Among the examples of strongly protomodular categories there are the categories of groups, Lie algebras, rings and, more generally, any ``category of interest'' in the sense of Orzech \cite{Orzech}, as shown in \cite{MaMe}.
Also the dual category of any elementary topos is strongly protomodular, as is the category of cocommutative Hopf algebras over a field. The semi-abelian variety of digroups is not strongly protomodular.

\section{Monoidal structure and completeness of the category $\HBR$}

It is well-known that $\HBR$ and $\CHBR$ are monoidal categories. Let us recall the monoidal structure.  \medskip

\noindent\textbf{Monoidal structure}. Let $(H,\cdot,\bullet,1,\Delta,\varepsilon,S,T)$ and $(K,\cdot',\bullet',1',\Delta',\varepsilon',S',T')$ be two Hopf braces. Since the category of Hopf algebras is monoidal with tensor product $\otimes$, we have that $H^{\cdot}\otimes K^{\cdot'}$ and $H^{\bullet}\otimes K^{\bullet'}$ are Hopf algebras on the vector space $H\otimes K$. The two multiplications are given by
\[
m_{\cdot_{\otimes}}:=(m_{\cdot}\otimes m_{\cdot'})(\mathrm{Id}_{H}\otimes\tau_{K,H}\otimes\mathrm{Id}_{K}),\qquad m_{\bullet_{\otimes}}:=(m_{\bullet}\otimes m_{\bullet'})(\mathrm{Id}_{H}\otimes\tau_{K,H}\otimes\mathrm{Id}_{K}),
\]
where $\tau$ is the canonical flip map, and the unit, which is the same for both the algebra structures, is given by $1_{\otimes}:=1\otimes1'$. The coalgebra structure of both the Hopf algebras $H^{\cdot}\otimes K^{\cdot'}$ and $H^{\bullet}\otimes K^{\bullet'}$ is given by
\[
\Delta_{\otimes}:=(\mathrm{Id}_{H}\otimes\tau_{H,K}\otimes\mathrm{Id}_{K})(\Delta\otimes\Delta'), \quad \varepsilon_{\otimes}:=\varepsilon\otimes\varepsilon'
\]
and the antipodes are 
\[
S_{\otimes}:=(S\otimes S'), \quad T_{\otimes}:=(T\otimes T').
\]
It is straightforward to prove that \eqref{eq:HBRcomp.} is satisfied:
\[
\begin{split}
    &\big((h\otimes k)_{1}\bullet_{\otimes}(a\otimes b)\big)\cdot_{\otimes}\big((S\otimes S')((h\otimes k)_{2})\big)\cdot_{\otimes}\big((h\otimes k)_{3}\bullet_{\otimes}(c\otimes d)\big)=\\&\big((h_{1}\bullet a)\otimes(k_{1}\bullet' b)\big)\cdot_{\otimes}\big(S(h_{2})\otimes S'(k_{2})\big)\cdot_{\otimes}\big((h_{3}\bullet c)\otimes(k_{3}\bullet' d)\big)=\\&\big((h_{1}\bullet a)\cdot S(h_{2})\cdot(h_{3}\bullet c)\big)\otimes\big( (k_{1}\bullet' b)\cdot' S'(k_{2})\cdot'(k_{3}\bullet' d)\big)=\\&\big(h\bullet(a\cdot c)\big)\otimes\big(k\bullet'(b\cdot' d)\big)=(h\otimes k)\bullet_{\otimes}((a\otimes b)\cdot_{\otimes}(c\otimes d)).
\end{split}
\]
It follows that
$(H\otimes K,m_{\cdot_{\otimes}},m_{\bullet_{\otimes}},1_{\otimes},\Delta_{\otimes},\varepsilon_{\otimes},S_{\otimes},T_{\otimes})$ is a Hopf brace.  
Moreover, 
the unit object is clearly given by $(\Bbbk,\cdot,\cdot)$. We denote the tensor product Hopf brace by $(H\otimes K,\cdot_{\otimes},\bullet_{\otimes})$. If $(H,\cdot,\bullet)$ and $(K,\cdot',\bullet')$ are cocommutative Hopf braces then also $(H\otimes K,\cdot_{\otimes},\bullet_{\otimes})$ is cocommutative. \medskip

The categories $\HBR$ and $\CHBR$ are  \emph{pointed}.  Indeed, given a Hopf algebra $H^{\cdot}$, the unit $u_{H^{\cdot}}$ is the unique morphism of Hopf algebras from $\Bbbk$ to $H^{\cdot}$ and the counit $\varepsilon_{H^{\cdot}}$ is the unique morphism of Hopf algebras from $H^{\cdot}$ to $\Bbbk$. Given a Hopf brace $(H,\cdot,\bullet)$, one has $u_{H^{\cdot}}=u_{H^{\bullet}}$ and $\varepsilon_{H^{\cdot}}=\varepsilon_{H^{\bullet}}$, hence $(\Bbbk,\cdot,\cdot)$ is the zero object in $\HBR$ and in $\CHBR$. Given two Hopf braces $(H,\cdot,\bullet)$ and $(K,\cdot,\bullet)$, the morphism $u_{K}\varepsilon_{H}$ is the zero morphism $(H,\cdot,\bullet)\to(K,\cdot,\bullet)$. \medskip

\noindent\textbf{Finite limits}. In \cite[Theorem 3.1 and Corollary 3.3]{Ag} it is shown that the category $\CHBR$ is complete and the category $\HBR$ has equalizers. As it is pointed out in \cite[Remark 3.3]{AC}, it is not known if $\HBR$ is complete. 

We'll now recall how equalizers in $\HBR$ (and $\CHBR$) and binary products in $\CHBR$ are made. For this, the following definition will be needed:

\begin{definition}
    Given a Hopf brace $(H,\cdot,\bullet,1,\Delta,\varepsilon,S,T)$ and a vector subspace $A\subseteq H$, we say that $(A,\cdot,\bullet)$ is a \textit{sub-Hopf brace} of $(H,\cdot,\bullet)$ if $A$ is closed in $H$ under $m_{\cdot}$, $m_{\bullet}$, $\Delta$, $S$, $T$, and it contains $1$.
\end{definition}

\begin{remark}
    The previous definition means that $(A,\cdot,1,\Delta,\varepsilon,S)$ and $(A,\bullet,1,\Delta,\varepsilon, T)$ are Hopf subalgebras of $(H,\cdot,1,\Delta,\varepsilon,S)$ and $(H,\bullet,1,\Delta,\varepsilon,T)$, respectively. Clearly $(A,\cdot,\bullet,1,\Delta,\varepsilon,S,T)$ is a Hopf brace since \eqref{eq:HBRcomp.} is satisfied for all $a,b,c\in H$.
\end{remark}

\noindent\textbf{Equalizers}. Let $f,g:(H,\cdot,\bullet)\to(A,\cdot,\bullet)$ be two morphisms in $\HBR$. Then, defining 
\[
K:=\{h\in H\ |\ h_{1}\otimes f(h_{2})\otimes h_{3}=h_{1}\otimes g(h_{2})\otimes h_{3}\}
\]
one has that $(K,\cdot,\bullet)$ is a sub-Hopf brace of $(H,\cdot,\bullet)$ and the inclusion $i:(K,\cdot,\bullet)\hookrightarrow (H,\cdot,\bullet)$ is the equalizer of $(f,g)$ in $\HBR$, see \cite[Example 3.2]{Ag}. In particular, the kernel of $f$ in $\HBR$ is given by $(\mathrm{Hker}(f),\cdot,\bullet)\subseteq(H,\cdot,\bullet)$, where
\[
\mathrm{Hker}(f):=\{h\in H\ |\ h_{1}\otimes f(h_{2})\otimes h_{3}=h_{1}\otimes 1_{A}\otimes h_{2}\}.
\]
If we consider the category $\CHBR$ of cocommutative Hopf braces, it is straightforward to prove that the equalizer of $f,g:(H,\cdot,\bullet)\to(A,\cdot,\bullet)$ in $\CHBR$ becomes
\[
K=\{h\in H\ |\ h_{1}\otimes f(h_{2})=h_{1}\otimes g(h_{2})\}.
\]
Hence, the kernel of $f$ in $\CHBR$ is given by
\[
\mathrm{Hker}(f)=\{h\in H\ |\ h_{1}\otimes f(h_{2})=h\otimes 1_{A}\}.
\]

\begin{remark}
   Let us observe that, given $f,g:(H,\cdot,\bullet)\to(A,\cdot,\bullet)$ in $\HBR$ (resp. $\CHBR$), the inclusions $i:K^{\cdot}\hookrightarrow H^{\cdot}$ and $i:K^{\bullet}\hookrightarrow H^{\bullet}$ are, respectively, the equalizers of the pairs $(f:H^{\cdot}\to A^{\cdot},g:H^{\cdot}\to A^{\cdot})$ and $(f:H^{\bullet}\to A^{\bullet},g:H^{\bullet}\to A^{\bullet})$ in $\mathsf{Hopf}_{\Bbbk}$ (resp. $\mathsf{Hopf}_{\Bbbk,\mathrm{coc}}$), see \cite{Ag2}. This just means that the forgetful functors $F^{\cdot},F^{\bullet}:\HBR\to\mathsf{Hopf}_{\Bbbk}$ (resp. $F^{\cdot},F^{\bullet}:\CHBR\to\mathsf{Hopf}_{\Bbbk,\mathrm{coc}}$) assigning to a (cocommutative) Hopf brace $(H,\cdot,\bullet)$ the (cocommutative) Hopf algebras $H^{\cdot}$ and $H^{\bullet}$, respectively, preserve equalizers.
\end{remark}

\noindent\textbf{Binary products}. Let us describe binary products only for the category $\CHBR$, keeping in mind that this construction only works in presence of cocommutativity.

Let $(H,\cdot,\bullet)$ and $(K,\cdot,\bullet)$ be two cocommutative Hopf braces. Then, their binary product in $\CHBR$ is given by  
\[
\big((H\otimes K,\cdot_{\otimes},\bullet_{\otimes}),\pi_{H},\pi_{K}\big)
\]
where $\pi_{H}:=\mathrm{Id}_{H}\otimes\varepsilon_{K}$ and $\pi_{K}:=\varepsilon_{H}\otimes\mathrm{Id}_{K}$. Let us observe that, given another cocommutative Hopf brace $(L,\cdot,\bullet)$ and any two morphisms of cocommutative Hopf braces $f:(L,\cdot,\bullet)\to(H,\cdot,\bullet)$ and $g:(L,\cdot,\bullet)\to(K,\cdot,\bullet)$, the unique morphism $\phi:(L,\cdot,\bullet)\to(H\otimes K,\cdot_{\otimes},\bullet_{\otimes})$ such that $ \pi_H \phi =f$ and $ \pi_K \phi =g$ is defined by $\phi(x)=(f\otimes g)\Delta_{L}(x)$. Note that $\phi$ is a morphism of coalgebras (and then a morphism of Hopf braces) because $\Delta_{L}$ is a morphism of coalgebras (since $L$ is cocommutative). 

\begin{remark}
The monoidal tensor product coincides with the binary product in $\CHBR$, and the monoidal unit with the terminal object. This means that the category $\CHBR$ is a \textit{cartesian} monoidal category.
\end{remark}

\begin{remark}
    Given two cocommutative Hopf braces $(H,\cdot,\bullet)$ and $(K,\cdot,\bullet)$, we have that $(H\otimes K)^{\cdot_{\otimes}}$ is the binary product of $H^{\cdot}$ and $K^{\cdot}$ in $\mathsf{Hopf}_{\Bbbk,\mathrm{coc}}$ as well as $(H\otimes K)^{\bullet_{\otimes}}$ is the product of $H^{\bullet}$ and $K^{\bullet}$ in $\mathsf{Hopf}_{\Bbbk,\mathrm{coc}}$. So, the forgetful functors $F^{\cdot},F^{\bullet}:\CHBR\to\mathsf{Hopf}_{\Bbbk,\mathrm{coc}}$ preserve also binary products. 
\end{remark}

Having the explicit description of equalizers and binary products in $\CHBR$ one can then build pullbacks in $\CHBR$. Clearly, as vector spaces, pullbacks in $\CHBR$ coincide with pullbacks in $\mathsf{Hopf}_{\Bbbk,\mathrm{coc}}$, computed with respect to both the Hopf algebra structures. \medskip

\noindent\textbf{Pullbacks}. Given $f:(A,\cdot,\bullet)\to(C,\cdot,\bullet)$ and $g:(B,\cdot,\bullet)\to(C,\cdot,\bullet)$ in $\CHBR$, the pullback object $A\times_{C}B$ in $\CHBR$ is given by the equalizer object of the pair $(f\circ\pi_{A},g\circ\pi_{B})$ in $\CHBR$, where $\pi_{A}=\mathrm{Id}_{A}\otimes\varepsilon_{B}$ and $\pi_{B}=\varepsilon_{A}\otimes\mathrm{Id}_{B}$. Thus, the ``object part'' $A\times_{C}B$ of the pullback is given by elements $a\otimes b\in A\otimes B$ such that
\[
(\mathrm{Id}_{A\otimes B}\otimes f(\mathrm{Id}_{A}\otimes\varepsilon_{B}))\Delta_{A\otimes B}(a\otimes b)=(\mathrm{Id}_{A\otimes B}\otimes g(\varepsilon_{A}\otimes\mathrm{Id}_{B}))\Delta_{A\otimes B}(a\otimes b),
\]
i.e. $a_{1}\otimes b\otimes f(a_{2})=a\otimes b_{1}\otimes g(b_{2})$. This equality is equivalent to $a_{1}\otimes f(a_{2})\otimes b=a\otimes g(b_{2})\otimes b_{1}$ and, by cocommutativity, also to $ a_{1}\otimes f(a_{2})\otimes b=a\otimes g(b_{1})\otimes b_{2}$. Thus, $A\times_{C}B$ is the vector space
\[
A\times_{C}B=\{a\otimes b\in A\otimes B\ |\ a_{1}\otimes f(a_{2})\otimes b=a\otimes g(b_{1})\otimes b_{2}\}
\]
and it is a sub-Hopf brace of the product $(A\otimes B,\cdot_{\otimes},\bullet_{\otimes})$. The pullback of the pair $(f,g)$ in $\CHBR$ is then given by $((A\times_{C}B,\cdot_{\otimes},\bullet_{\otimes}),\pi_{A},\pi_{B})$.

\section{The Split Short Five Lemma holds in $\CHBR$}\label{Section-protomodular} 
The category $\mathsf{Hopf}_{\Bbbk,\mathrm{coc}}$ of cocommutative Hopf algebras is protomodular, since so is any category of internal groups in a category with finite limits (see \cite[Example 5, p. 57]{Bourn-90}), and $\mathsf{Hopf}_{\Bbbk,\mathrm{coc}}$ can be seen as the category $\mathsf{Grp}(\mathsf{Coalg}_{\Bbbk,\mathrm{coc}})$ of internal groups in the category of cocommutative coalgebras. 

For the sake of completeness we'll prove the protomodularity of $\CHBR$ by showing that the \emph{Split Short Five Lemma} holds true in it. \medskip


Let us first recall that, given a Hopf algebra $H$ and an object $R$ in $\mathsf{Hopf}(_{H}\mathfrak{M})$, one can build the \textit{smash product} Hopf algebra $R\#H$ (also called \emph{semi-direct product}) which is $R\otimes H$ as vector space, with product given by
\[
(r\otimes h)\cdot_{\#}(r'\otimes h'):=r(h_{1}\rightharpoonup r')\otimes h_{2}h'
\]
and unit $1_{\#}:=1_{R}\otimes 1_{H}$, where $\rightharpoonup:H\otimes R\to R$ denotes the $H$-action on $R$. Moreover, one defines $\varepsilon_{\#}:=\varepsilon_{R}\otimes\varepsilon_{H}$ and $\Delta_{\#}(r\otimes h)=r_{1}\otimes h_{1}\otimes r_{2}\otimes h_{2}$, while the antipode is $S_{\#}(r\otimes h):=(S_{H}(h_{1})\rightharpoonup S_{R}(r))\otimes S_{H}(h_{2})$. We also recall that, given a pair of morphisms
\begin{tikzcd}
	A & H
	\arrow[shift left, from=1-1, to=1-2,"\pi"]
	\arrow[shift left, from=1-2, to=1-1,"\gamma"]
\end{tikzcd}
of cocommutative Hopf algebras such that $\pi\circ\gamma=\mathrm{Id}_{H}$, 
there is an isomorphism of Hopf algebras 
\[
\phi:\mathrm{Hker}(\pi)\# H\to A,\quad k\otimes h\mapsto k\gamma(h), 
\]
see \cite{Molnar}, 
where $\mathrm{Hker}(\pi)$ is in $\mathsf{Hopf}(_{H}\mathfrak{M})$ with action $\rightharpoonup:H\otimes\mathrm{Hker}(\pi)\to\mathrm{Hker}(\pi)$ given by the assignment $h\otimes k\mapsto\gamma(h_{1})k\gamma(S_{H}(h_{2}))$. The inverse of $\phi$ is given by $\phi^{-1}:A\to\mathrm{Hker}(\pi)\#H$, $a\mapsto a_{1}\gamma(\pi(S(a_{2})))\otimes\pi(a_{3})$.

\begin{proposition}\label{prop:protomodularity}
    The Split Short Five Lemma holds in $\CHBR$. Accordingly, the finitely complete pointed category $\CHBR$ is protomodular.
\end{proposition}

\begin{proof}
Consider the following commutative diagram
\[
\begin{tikzcd}
	(\mathrm{Hker}(\pi),\cdot,\bullet) & (A,\cdot,\bullet) & (H,\cdot,\bullet) \\
	(\mathrm{Hker}(\pi'),\cdot,\bullet) & (A',\cdot,\bullet) & (H',\cdot,\bullet)
	\arrow[hook,from=1-1, to=1-2,"\iota"]
	\arrow[from=1-1, to=2-1, "f"']
	\arrow[shift left, from=1-2, to=1-3,"\pi"]
	\arrow[from=1-2, to=2-2, "g"]
	\arrow[shift left, from=1-3, to=1-2, "\gamma"]
	\arrow[from=1-3, to=2-3,"l"]
	\arrow[hook,from=2-1, to=2-2, "\iota'"]
	\arrow[shift left, from=2-2, to=2-3, "\pi'"]
	\arrow[shift left, from=2-3, to=2-2, "\gamma'"]
\end{tikzcd}
\]
in $\CHBR$, where $\iota:(\mathrm{Hker}(\pi),\cdot,\bullet)\hookrightarrow(A,\cdot,\bullet)$ and $\iota':(\mathrm{Hker}(\pi'),\cdot,\bullet)\hookrightarrow(A',\cdot,\bullet)$ are the kernels of $\pi$ and $\pi'$ in $\CHBR$, respectively, $\pi\circ\gamma=\mathrm{Id}_{H}$ and $\pi'\circ\gamma'=\mathrm{Id}_{H'}$.

Suppose that $f$ and $l$ are isomorphisms in $\CHBR$ and $g$ is a morphism in $\CHBR$. As observed above, we know that $\phi:\mathrm{Hker}(\pi)^{\cdot}\#H^{\cdot}\to A^{\cdot}$, defined by $k\otimes h\mapsto k\cdot\gamma(h)$, and $\phi':\mathrm{Hker}(\pi')^{\cdot}\#H'^{\cdot}\to A'^{\cdot}$, defined by $k\otimes h\mapsto k\cdot\gamma'(h)$, are isomorphisms of Hopf algebras. We can compute
\[
g\phi(k\otimes h)=g(k\cdot\gamma(h))=g(k)\cdot g(\gamma(h))=f(k)\cdot\gamma'(l(h))=\phi'(f\otimes l)(k\otimes h),
\]
i.e. $g\circ\phi=\phi'\circ(f\otimes l)$. Since $\phi$, $\phi'$ and $f\otimes l$ are bijective linear maps, also $g$ is a bijective linear map, hence $g$ is an isomorphism in $\CHBR$.
\end{proof}

    We have included a direct proof of the protomodularity of $\CHBR$ for the reader's convenience. This same result 
    could also be deduced from the protomodularity of the category of cocommutative Hopf algebras. Indeed, the forgetful functors $F^{\cdot},F^{\bullet}:\CHBR\to\mathsf{Hopf}_{\Bbbk,\mathrm{coc}}$ preserve (finite) limits and reflect isomorphisms, hence $\CHBR$ is protomodular since $\mathsf{Hopf}_{\Bbbk,\mathrm{coc}}$ is so.

\noindent\textbf{Semi-direct product decomposition for cocommutative Hopf braces}. Here we provide a semi-direct product decomposition for cocommutative Hopf braces, that will also be used later. First we show the following result, whose proof is straightforward.

\begin{lemma}\label{lem:transportingHopfbraces}
     Let $(K,\cdot,\bullet)$ be a Hopf brace. 
\begin{itemize}
    \item[1)] If $H^{\cdot}$ is a Hopf algebra and $f:H^{\cdot}\to K^{\cdot}$ is an isomorphism of Hopf algebras, then there is a unique Hopf brace structure $(H,\cdot,\overline{\bullet})$ such that $f:(H,\cdot,\overline{\bullet})\to(K,\cdot,\bullet)$ is an isomorphism in $\HBR$. \medskip
    \item[2)] If $H^{\bullet}$ is a Hopf algebra and $f:H^{\bullet}\to K^{\bullet}$ is an isomorphism of Hopf algebras, then there is a unique Hopf brace structure $(H,\overline{\cdot},\bullet)$ such that $f:(H,\overline{\cdot},\bullet)\to(K,\cdot,\bullet)$ is an isomorphism in $\HBR$.
\end{itemize}  
    
\end{lemma}

\begin{proof}
    1). Given $a,b\in H$, we must define $a\overline{\bullet} b:=f^{-1}(f(a)\bullet f(b))$ in order to have $f(a\overline{\bullet} b)=f(a)\bullet f(b)$. In fact one can transport the algebra structure of $K^{\bullet}$ through $f$ such that $f:(H,\overline{\bullet},1)\to (K,\bullet,1)$ becomes an isomorphism of algebras. Hence, if we show that $(H,\cdot,\overline{\bullet})$ is an object in $\HBR$, then $f$ will be automatically an isomorphism in $\HBR$. 

We already know that $H$ is a coalgebra, we need to prove that $H^{\overline{\bullet}}$ is a bialgebra. Since $f$ is a morphism of coalgebras, the same holds for $f^{-1}$ and then $H^{\overline{\bullet}}$ is a bialgebra because $K^{\bullet}$ is a bialgebra. Indeed, we obtain
\[
\varepsilon(a\overline{\bullet} b)=\varepsilon(f^{-1}(f(a)\bullet f(b)))=\varepsilon(f(a)\bullet f(b))=\varepsilon(f(a))\varepsilon(f(b))=\varepsilon(a)\varepsilon(b)
\]
and 
\[
\begin{split}
\Delta(a\overline{\bullet} b)&=\Delta(f^{-1}(f(a)\bullet f(b)))=(f^{-1}\otimes f^{-1})\Delta(f(a)\bullet f(b))\\&=f^{-1}(f(a)_{1}\bullet f(b)_{1})\otimes f^{-1}(f(a)_{2}\bullet f(b)_{2})\\&=f^{-1}(f(a_{1})\bullet f(b_{1}))\otimes f^{-1}(f(a_{2})\bullet f(b_{2}))\\&=(a_{1}\overline{\bullet} b_{1})\otimes(a_{2}\overline{\bullet} b_{2}).
\end{split}
\]
Moreover, $H^{\overline{\bullet}}$ is a Hopf algebra with antipode $f^{-1}Tf$ since $K^{\bullet}$ has antipode $T$ and $f^{-1}(1)=1$. Indeed, we have
\[
\begin{split}
(f^{-1}Tf(a_{1}))\overline{\bullet} a_{2}&=f^{-1}\big(ff^{-1}Tf(a_{1})\bullet f(a_{2})\big)=f^{-1}(T(f(a)_{1})\bullet f(a)_{2})\\&=f^{-1}(\varepsilon(f(a))1)=\varepsilon(a)f^{-1}(1)=\varepsilon(a)1,
\end{split}
\]
analogously for the other one. It just remains to prove that \eqref{eq:HBRcomp.} holds true for $(H,\cdot,\overline{\bullet})$. This happens since \eqref{eq:HBRcomp.} holds for $(K,\cdot,\bullet)$ and $f$ is a morphism of coalgebras and algebras with respect to the product $\cdot$ (and hence also $f^{-1}$ is so). Indeed, we have
\[
\begin{split}
    a\overline{\bullet}(b\cdot c)&=f^{-1}(f(a)\bullet f(b\cdot c))=f^{-1}(f(a)\bullet(f(b)\cdot f(c)))\\&=f^{-1}\Big(\big(f(a)_{1}\bullet f(b)\big)\cdot S(f(a)_{2})\cdot\big(f(a)_{3}\bullet f(c)\big)\Big)\\&=f^{-1}\Big(\big(f(a_{1})\bullet f(b)\big)\cdot f(S(a_{2}))\cdot\big(f(a_{3})\bullet f(c)\big)\Big)\\&=f^{-1}\big(f(a_{1})\bullet f(b)\big)\cdot f^{-1}f(S(a_{2}))\cdot f^{-1}\big(f(a_{3})\bullet f(c)\big)\\&=(a_{1}\overline{\bullet} b)\cdot S(a_{2})\cdot(a_{3}\overline{\bullet} c).
\end{split}
\]

\noindent 2). We must define $a\overline{\cdot} b:=f^{-1}(f(a)\cdot f(b))$ in order to have $f(a\overline{\cdot} b)=f(a)\cdot f(b)$, for all $a,b\in H$. Analogously to what was done before, $(H,\overline{\cdot},1,\Delta,\varepsilon,\overline{S}:=f^{-1}Sf)$ is the unique Hopf algebra structure such that $f:H^{\overline{\cdot}}\to K^{\cdot}$ is an isomorphism of Hopf algebras. We only have to verify that \eqref{eq:HBRcomp.} is satisfied for $(H,\overline{\cdot},\bullet)$:
\[
\begin{split}
    a\bullet(b\overline{\cdot}c)&=a\bullet f^{-1}(f(b)\cdot f(c))=f^{-1}\big(f(a)\bullet(f(b)\cdot f(c))\big)\\&=f^{-1}\big((f(a)_{1}\bullet f(b))\cdot S(f(a)_{2})\cdot(f(a)_{3}\bullet f(c))\big)\\&=f^{-1}\big((f(a_{1})\bullet f(b))\cdot S(f(a_{2}))\cdot(f(a_{3})\bullet f(c))\big)\\&=f^{-1}\big(f(a_{1}\bullet b)\cdot ff^{-1}Sf(a_{2})\cdot f(a_{3}\bullet c)\big)\\&=(a_{1}\bullet b)\overline{\cdot}f^{-1}Sf(a_{2})\overline{\cdot}(a_{3}\bullet c)\\&=(a_{1}\bullet b)\overline{\cdot}\overline{S}(a_{2})\overline{\cdot}(a_{3}\bullet c)
\end{split}
\]
using that \eqref{eq:HBRcomp.} holds for $(K,\cdot,\bullet)$ and $f$ is a morphism of coalgebras and algebras with respect to $\bullet$ (and so is $f^{-1}$).
\end{proof}

The previous result allows us to obtain the aforementioned semi-direct product decomposition for cocommutative Hopf braces.

\begin{proposition}\label{thm:points}
Let 
\[
\begin{tikzcd}
	(A,\cdot,\bullet) & (H,\cdot,\bullet)
	\arrow[shift left, from=1-1, to=1-2, "\pi"]
	\arrow[shift left, from=1-2, to=1-1, "\gamma"]
\end{tikzcd}
\]
be an object in the category $\mathsf{Pt}(\CHBR)$ of ``points'' in $\CHBR$, i.e. $\pi$ and $\gamma$ are morphisms in $\CHBR$ such that $\pi\circ\gamma=\mathrm{Id}_{H}$. Then, $(A,\cdot,\bullet)$ is isomorphic to $(\mathrm{Hker}(\pi)\otimes H,\cdot_{\#},\overline{\bullet})$ in $\CHBR$, where
\[
(k\otimes h)\cdot_{\#}(k'\otimes h'):=\big(k\cdot\gamma(h_{1})\cdot k'\cdot\gamma S(h_{2})\big)\otimes( h_{3}\cdot h')
\]
and
\[
(k\otimes h)\overline{\bullet}(k'\otimes h'):=\big((k\cdot\gamma(h_{1}))\bullet(k'\cdot\gamma(h'_{1}))\big)\cdot\gamma S(h_{2}\bullet h'_{2})\otimes\big( h_{3}\bullet h'_{3}\big).
\]
Moreover, $(A,\cdot,\bullet)$ is also isomorphic to $(\mathrm{Hker}(\pi)\otimes H,\overline{\cdot},\bullet_{\#})$ in $\CHBR$, where
\[
(k\otimes h)\bullet_{\#}(k'\otimes h'):=\big(k\bullet\gamma(h_{1})\bullet k'\bullet\gamma T(h_{2})\big)\otimes( h_{3}\bullet h')
\]
and
\[
(k\otimes h)\overline{\cdot}(k'\otimes h'):=\big((k\bullet\gamma(h_{1}))\cdot(k'\bullet\gamma(h'_{1}))\big)\bullet\gamma T(h_{2}\cdot h'_{2})\otimes\big( h_{3}\cdot h'_{3}\big).
\]
\end{proposition}

\begin{proof}
As we recalled in Section \ref{Section-protomodular}, there is an isomorphism of Hopf algebras 
\[
\phi^{\cdot}:(\mathrm{Hker}(\pi)\otimes H)^{\cdot_{\#}}\to A^{\cdot},\quad k\otimes h\mapsto k\cdot\gamma(h),
\]
where the product $\cdot_{\#}$ of $\mathrm{Hker}(\pi)\otimes H$ is defined by 
\[
(k\otimes h)\cdot_{\#}(k'\otimes h'):=\big(k\cdot(h_{1}\rightharpoonup k')\big)\otimes(h_{2}\cdot h')=\big(k\cdot\gamma(h_{1})\cdot k'\cdot\gamma(S(h_{2}))\big)\otimes( h_{3}\cdot h').
\]
The inverse of $\phi^{\cdot}$ is given by $$(\phi^{\cdot})^{-1}:A^{\cdot}\to(\mathrm{Hker}(\pi)\otimes H)^{\cdot_{\#}}, \quad  a\mapsto a_{1}\cdot\gamma(\pi(S(a_{2})))\otimes\pi(a_{3}).$$ By 1) of Lemma \ref{lem:transportingHopfbraces} there is a unique Hopf brace structure $(\mathrm{Hker}(\pi)\otimes H,\cdot_{\#},\overline{\bullet})$ such that $\phi^{\cdot}:(\mathrm{Hker}(\pi)\otimes H,\cdot_{\#},\overline{\bullet})\to(A,\cdot,\bullet)$ is an isomorphism in $\CHBR$. Explicitly, we have
\[
\begin{split}
    &(k\otimes h)\overline{\bullet}(k'\otimes h'):=(\phi^{\cdot})^{-1}\big(\phi^{\cdot}(k\otimes h)\bullet\phi^{\cdot}(k'\otimes h')\big)=(\phi^{\cdot})^{-1}\big((k\cdot\gamma(h))\bullet(k'\cdot\gamma(h'))\big)\\&=\big((k_{1}\cdot\gamma(h_{1}))\bullet(k'_{1}\cdot\gamma(h'_{1}))\big)\cdot S\gamma\pi((k_{2}\cdot\gamma(h_{2}))\bullet(k'_{2}\cdot\gamma(h'_{2})))\otimes\pi\big((k_{3}\cdot\gamma(h_{3}))\bullet(k'_{3}\cdot\gamma(h'_{3}))\big)\\&=\big((k_{1}\cdot\gamma(h_{1}))\bullet(k'_{1}\cdot\gamma(h'_{1}))\big)\cdot S\gamma((\pi(k_{2})\cdot h_{2})\bullet(\pi(k'_{2})\cdot h'_{2}))\otimes(\pi(k_{3})\cdot h_{3})\bullet(\pi(k'_{3})\cdot h'_{3})\\&\overset{(!)}{=}\big((k_{1}\cdot\gamma(h_{1}))\bullet(k'_{1}\cdot\gamma(h'_{1}))\big)\cdot S\gamma((\pi(k_{2})\cdot h_{2})\bullet(\pi(k'_{2})\cdot h'_{2}))\otimes\big( h_{3}\bullet h'_{3}\big)\\&\overset{(!)}{=}\big((k\cdot\gamma(h_{1}))\bullet(k'\cdot\gamma(h'_{1}))\big)\cdot\gamma(S(h_{2}\bullet h'_{2}))\otimes\big( h_{3}\bullet h'_{3}\big)
\end{split}
\]
where $(!)$ follows from the fact that $k,k'\in\mathrm{Hker}(\pi)$. 

We can also consider the isomorphism of Hopf algebras
\[
\phi^{\bullet}:(\mathrm{Hker}(\pi)\otimes H)^{\bullet_{\#}}\to A^{\bullet},\quad k\otimes h\mapsto k\bullet\gamma(h),
\]
where the product $\bullet_{\#}$ of $\mathrm{Hker}(\pi)\otimes H$ is defined by 
\[
(k\otimes h)\bullet_{\#}(k'\otimes h'):=\big(k\bullet(h_{1}\rightharpoonup k')\big)\otimes(h_{2}\bullet h')=\big(k\bullet\gamma(h_{1})\bullet k'\bullet\gamma(T(h_{2}))\big)\otimes( h_{3}\bullet h').
\]
The inverse of $\phi^{\bullet}$ is given by $(\phi^{\bullet})^{-1}:A^{\bullet}\to(\mathrm{Hker}(\pi)\otimes H)^{\bullet_{\#}}$, $a\mapsto a_{1}\bullet\gamma(\pi(T(a_{2})))\otimes\pi(a_{3})$. Then, by 2) of Lemma \ref{lem:transportingHopfbraces}, there exists a unique Hopf brace structure $(\mathrm{Hker}(\pi)\otimes H,\overline{\cdot},\bullet_{\#})$ such that $\phi^{\bullet}:(\mathrm{Hker}(\pi)\otimes H,\overline{\cdot},\bullet_{\#})\to(A,\cdot,\bullet)$ is an isomorphism in $\CHBR$. Similar computations to the ones above show that the product $\overline{\cdot}$, which is defined by $(k\otimes h)\overline{\cdot}(k'\otimes h'):=(\phi^{\bullet})^{-1}\big(\phi^{\bullet}(k\otimes h)\cdot\phi^{\bullet}(k'\otimes h')\big)$, is the one in the last statement of the Proposition.
\end{proof}

A deeper study of semi-direct products in $\CHBR$ could be undertaken. Indeed, in view of the semi-abelianness of $\CHBR$ obtained in Theorem \ref{thm:semi-ab}, one could study the categorical notions of internal action and semi-direct product for $\CHBR$. We leave this as a possible future project.

\section{Regularity of the category $\CHBR$}

In this section we show that the (finitely) complete category $\CHBR$ is regular.
For this, we need the description of cokernels of kernels in $\CHBR$. It will be shown that a cokernel of the kernel of a morphism $f:(A,\cdot,\bullet)\to(B,\cdot,\bullet)$ in $\CHBR$ is just the canonical projection  $\pi:(A,\cdot,\bullet)\to(A/\mathrm{ker}(f),\overline{\cdot},\overline{\bullet})$, where $\mathrm{ker}(f)$ denotes the vector space kernel of $f$. This can also be seen as a special case of the general construction of coequalizers given in \cite{AC}.


In order to consider  quotients in the category $\CHBR$ we introduce the notion of “ideal” for Hopf braces. 
\begin{definition}\label{def:Hopfbraceideal}
    Let $(H,\cdot,\bullet)$ be an object in $\HBR$. A \textit{Hopf brace ideal} is a vector subspace $I\subseteq H$ such that:
\begin{itemize}
    \item[1)] $H\cdot I\subseteq I$, $I\cdot H\subseteq I$, \quad $H\bullet I\subseteq I$, $I\bullet H\subseteq I$, \medskip
    \item[2)] $\Delta(I)\subseteq I\otimes H+H\otimes I$, $\varepsilon(I)=0$, \medskip
    \item[3)] $S(I)\subseteq I$, $T(I)\subseteq I$.
\end{itemize}
\end{definition}

\begin{remark}
   The properties in the previous definition mean that $I$ is a Hopf ideal with respect to both $H^{\cdot}$ and $H^{\bullet}$, hence the linear quotient $H/I$ has two Hopf algebra structures $\overline{\cdot}$ and $\overline{\bullet}$ induced from $\cdot$ and $\bullet$, respectively. The Hopf brace compatibility \eqref{eq:HBRcomp.} is induced from that of $H$, hence $(H/I,\overline{\cdot},\overline{\bullet})$ is a Hopf brace. 
\end{remark}

Recall that, given a cocommutative Hopf algebra $H$, there is a bijective correspondence due to K. Newman \cite[Theorem 4.1]{Newman} between the set of Hopf subalgebras of $H$ and the set of left ideals which are also two-sided coideals of $H$. More precisely, given a Hopf subalgebra $K$ of $H$ and a left ideal, two-sided coideal $I$ of $H$ the two inverse maps are given by
\[
\phi_{H}:K\mapsto HK^{+},\qquad  \psi_{H}:I\mapsto H^{\mathrm{co}\frac{H}{I}}:=\{x\in H\ |\ (\mathrm{Id}_{H}\otimes\pi)\Delta_{H}(x)=x\otimes\pi(1_{H})\},
\]
where $\pi:H\to H/I$ is the canonical projection. Given a morphism $f:A\to B$ in $\mathsf{Hopf}_{\Bbbk,\mathrm{coc}}$, the previous bijection is used to deduce that $A\mathrm{Hker}(f)^{+}=\mathrm{ker}(f)$, where $\mathrm{ker}(f)$ denotes the vector space kernel of $f$. Then, one obtains that $\mathsf{coker}(\mathsf{ker}(f))$ in $\mathsf{Hopf}_{\Bbbk,\mathrm{coc}}$ is given by the projection $A\to A/\mathrm{ker}(f)$.

We first show the following result, which will be very useful in the following.

\begin{lemma}\label{lem:normalsubHopfbrace}
    Given $(A,\cdot,\bullet)$ in $\CHBR$ and a sub-Hopf brace $(B,\cdot,\bullet)\subseteq(A,\cdot,\bullet)$ such that $A\rightharpoonup B\subseteq B$, then $\phi_{A^{\cdot}}(B):=A\cdot B^{+}=A\bullet B^{+}=:\phi_{A^{\bullet}}(B)$. 
\end{lemma}

\begin{proof}
Consider an element $a\cdot b\in A\cdot B^{+}$. Then, since $A\rightharpoonup B\subseteq B$, we have $a\cdot b=a_{1}\bullet(T(a_{2})\rightharpoonup b)\in A\bullet B$. Moreover, $\varepsilon$ is left linear with respect to $\rightharpoonup$ and then $a\cdot b\in A\bullet B^{+}$. For the same reasons, given $a\bullet b\in A\bullet B^{+}$, we have $a\bullet b=a_{1}\cdot(a_{2}\rightharpoonup b)\in A\cdot B^{+}$. It follows that $A\cdot B^{+}=A\bullet B^{+}$. 
\end{proof}

We give the following definition.

\begin{definition}\label{def:normalHopfbrace}
    A sub-Hopf brace $(B,\cdot,\bullet)$ of a Hopf brace $(A,\cdot,\bullet)$ is \textit{normal} if: 
\begin{equation}\label{eq:normaldot}
a_{1}\cdot b\cdot S(a_{2})\in B, \qquad \text{for all}\ a\in A,b\in B,
\end{equation}
\begin{equation}\label{eq:normalbullet}
    a_{1}\bullet b\bullet T(a_{2})\in B,\qquad \text{for all}\ a\in A,b\in B,
\end{equation}
\begin{equation}\label{closureunderharpoon}
    a\rightharpoonup b\in B,\qquad \text{for all}\ a\in A, b\in B. 
\end{equation}
\end{definition}

\begin{remark}
A sub-Hopf brace $(B,\cdot,\bullet)$ of $(A,\cdot,\bullet)$ is normal if $B^{\cdot}$ and $B^{\bullet}$ are normal Hopf subalgebras of $A^{\cdot}$ and $A^{\bullet}$, respectively, and $A\rightharpoonup B\subseteq B$.
\end{remark}

\begin{proposition}\label{prop:normalmono}
    Given a sub-Hopf brace $(B,\cdot,\bullet)\subseteq(A,\cdot,\bullet)$ in $\CHBR$ the following are equivalent:
\begin{itemize}
    \item[1)] $(B,\cdot,\bullet)$ is a normal sub-Hopf brace, i.e. \eqref{eq:normaldot}, \eqref{eq:normalbullet} and \eqref{closureunderharpoon} are satisfied; \medskip
    \item[2)] $A\cdot B^{+}=A\bullet B^{+}$ is a Hopf brace ideal; \medskip
    \item[3)] $(B,\cdot,\bullet)$ is the kernel of a morphism in $\CHBR$.
\end{itemize}
\end{proposition}

\begin{proof}
    $1)\implies2)$. From \eqref{eq:normaldot} and \eqref{eq:normalbullet}, using \cite[Corollary 2.3]{GSV}, we obtain, respectively, that $A\cdot B^{+}$ is a Hopf ideal of $A^{\cdot}$ and $A\bullet B^{+}$ is a Hopf ideal of $A^{\bullet}$. Moreover from \eqref{closureunderharpoon}, using Lemma \ref{lem:normalsubHopfbrace}, we obtain that $A\cdot B^{+}=A\bullet B^{+}$, which is then a Hopf brace ideal. \medskip
    
\noindent $2)\implies3)$. Suppose that $I:=A\cdot B^{+}=A\bullet B^{+}$ is a Hopf brace ideal. Hence, using \cite[Corollary 2.3]{GSV} we know that $B^{\cdot}=\mathrm{Hker}(\pi)$ with $\pi:A\to A/A\cdot B^{+}$ in $\mathsf{Hopf}_{\Bbbk,\mathrm{coc}}$ and $B^{\bullet}=\mathrm{Hker}(\pi')$ with $\pi':A^{\bullet}\to A^{\bullet}/A\bullet B^{+}$. But $A\cdot B^{+}=A\bullet B^{+}$ by assumption, so $\pi=\pi'$ and $(B,\cdot,\bullet)$ is the kernel of $\pi:(A,\cdot,\bullet)\to(A/I,\overline{\cdot},\overline{\bullet})$ in $\CHBR$. \medskip
    
\noindent $3)\implies1)$. Given a morphism $f:(A,\cdot,\bullet)\to(B,\cdot,\bullet)$ in $\CHBR$ and its kernel in $\CHBR$, which is given by the inclusion $j:(\mathrm{Hker}(f),\cdot,\bullet)\hookrightarrow(A,\cdot,\bullet)$, by \cite[Corollary 2.3]{GSV} we already know that \eqref{eq:normaldot} and \eqref{eq:normalbullet} are satisfied because $j:\mathrm{Hker}(f)^{\cdot}\hookrightarrow A^{\cdot}$ and $j:\mathrm{Hker}(f)^{\bullet}\hookrightarrow A^{\bullet}$ are kernels in $\mathsf{Hopf}_{\Bbbk,\mathrm{coc}}$. We are now going to show that $A\rightharpoonup\mathrm{Hker}(f)\subseteq\mathrm{Hker}(f)$ to conclude that $(\mathrm{Hker}(f) ,\cdot,\bullet)$ is normal. Given $a\in A$ and $b\in\mathrm{Hker}(f)$, we prove that $a\rightharpoonup b:=S(a_{1})\cdot(a_{2}\bullet b)\in\mathrm{Hker}(f)$:
\[
\begin{split}
(\mathrm{Id}\otimes f)\Delta(a\rightharpoonup b)&=\big(S(a_{1})_{1}\cdot(a_{2}\bullet b_{1})\big)\otimes f(S(a_{1})_{2}\cdot(a_{3}\bullet b_{2}))\\&
=\big(S(a_{2})\cdot(a_{3}\bullet b_{1})\big)\otimes f(S(a_{1}))\cdot(f(a_{4})\bullet f(b_{2}))\\&\overset{(!)}{=}\big(S(a_{2})\cdot(a_{3}\bullet b)\big)\otimes f(S(a_{1}))\cdot f(a_{4})\\&\overset{(!!)}{=}\big(S(a_{1})\cdot(a_{2}\bullet b)\big)\otimes f(S(a_{3})\cdot a_{4})\\&=(a\rightharpoonup b)\otimes 1, 
\end{split}
\]
where $(!)$ follows from the fact that $b\in\mathrm{Hker}(f)$ and $(!!)$ follows by cocommutativity. 
\end{proof}

\begin{proposition}\label{prop:factorization}
    Any morphism in $\CHBR$ factorizes as a normal epimorphism followed by a monomorphism. 
\end{proposition}

\begin{proof}
Consider a morphism $f:(A,\cdot,\bullet)\to(B,\cdot,\bullet)$ in $\CHBR$ and the kernel of $f$ in $\CHBR$, given by the inclusion $j:(\mathrm{Hker}(f),\cdot,\bullet)\hookrightarrow(A,\cdot,\bullet)$. By Proposition \ref{prop:normalmono} we have that $A\cdot\mathrm{Hker}(f)^{+}=A\bullet\mathrm{Hker}(f)^{+}$. Moreover, by \cite[Theorem 4.1]{Newman}, we know that the latter coincides with $\mathrm{ker}(f)$, the vector space kernel of $f$. 
We now show that the cokernel of $j:(\mathrm{Hker}(f),\cdot,\bullet)\hookrightarrow(A,\cdot,\bullet)$ in $\CHBR$ is given by the canonical projection $p:(A,\cdot,\bullet)\to(A/\mathrm{ker}(f),\overline{\cdot},\overline{\bullet})$. Indeed, $p$ coequalizes the pair of morphisms $(j,u_{A}\varepsilon_{\mathrm{Hker}(f)})$ in $\CHBR$. Suppose now that there exists a morphism $h:(A,\cdot,\bullet)\to(C,\cdot,\bullet)$ in $\CHBR$ such that $h\circ j=h\circ u_{A}\circ\varepsilon_{\mathrm{Hker}(f)}$. 
Since $p:A^{\cdot}\to A/\mathrm{ker}(f)^{\overline{\cdot}}$ is the cokernel of $j:\mathrm{Hker}(f)^{\cdot}\hookrightarrow A^{\cdot}$ in $\mathsf{Hopf}_{\Bbbk,\mathrm{coc}}$, there exists a unique morphism $q:A/\mathrm{ker}(f)^{\overline{\cdot}}\to C^{\cdot}$ in $\mathsf{Hopf}_{\Bbbk,\mathrm{coc}}$ such that $q\circ p=h$. Note that $p$ and $h$ are also morphisms of algebras with respect to $\bullet$, hence $q$ is a morphism in $\CHBR$, and the unique one such that $q\circ p=h$; so $p$ is the cokernel of $j$ in $\CHBR$. 
In particular, we obtain the following commutative diagram:
\[
\begin{tikzcd}
	(\mathrm{Hker}(f),\cdot,\bullet) & (A,\cdot,\bullet) & (\frac{A}{\mathrm{ker}(f)},\overline{\cdot},\overline{\bullet}) \\ 
	&& (B,\cdot,\bullet)
	\arrow[hook, from=1-1, to=1-2, "j"]
	\arrow[from=1-2, to=1-3, "p"]
	\arrow[from=1-2, to=2-3, "f"']
	\arrow[dotted, from=1-3, to=2-3, "i"]
\end{tikzcd}
\]
Any morphism $f$ in $\CHBR$ factorizes as $f=i\circ p$, with $p$ a normal epimorphism in $\CHBR$ and $i$ a monomorphism (an injective morphism) in $\CHBR$.
\end{proof}

We can then show the following result:

\begin{corollary}\label{lem:regepi}
    In the category $\CHBR$ of cocommutative Hopf braces:
\begin{itemize}
    \item[1)] regular epimorphisms coincide  with normal epimorphisms and with \\surjective morphisms; \medskip
    \item[2)] monomorphisms coincide with injective morphisms.
\end{itemize}
    
\end{corollary}

\begin{proof}
In a pointed finitely complete protomodular category, regular epimorphisms and normal epimorphisms coincide (see \cite{Bourn-90}). By the proof of Proposition \ref{prop:factorization} we know that cokernels of kernels in $\CHBR$ are projections, hence surjective maps. On the other hand, if $f$ in $\CHBR$ is surjective, using the factorization $f=i\circ p$ in $\CHBR$ in Proposition \ref{prop:factorization}, we get that $i$ is surjective, hence an isomorphism in $\CHBR$. 

An injective morphism in $\CHBR$ is a monomorphism in $\CHBR$. We now prove the converse. Let $f:(A,\cdot,\bullet)\to(B,\cdot,\bullet)$ be a monomorphism in $\CHBR$ and $j:(\mathrm{Hker}(f),\cdot,\bullet)\hookrightarrow(A,\cdot,\bullet)$ the kernel of $f$ in $\CHBR$. Then the equality $f\circ j=u_{B}\circ\varepsilon_{A}\circ j=f\circ u_{A}\circ\varepsilon_{A}\circ j$ implies that $j=u_{A}\circ\varepsilon_{A}\circ j$,
since $f$ is a monomorphism in $\CHBR$. This means that $\mathrm{Hker}(f)=\Bbbk1_{A}$ and then $\mathrm{Hker}(f)^{+}=0$. Accordingly, $\mathrm{ker}(f) = A\cdot\mathrm{Hker}(f)^{+}=0$, and $f$ is injective.
\end{proof}

We recall that a functor is said to be \textit{regular} if it preserves regular epimorphisms and finite limits.

\begin{lemma}\label{lem:regfunctor}
    The forgetful functors $F^{\cdot},F^{\bullet}:\CHBR\to\mathsf{Hopf}_{\Bbbk,\mathrm{coc}}$ are regular. Moreover, they reflect regular epimorphisms. 
\end{lemma}

\begin{proof}
    We already know that the forgetful functors $F^{\cdot},F^{\bullet}:\CHBR\to\mathsf{Hopf}_{\Bbbk,\mathrm{coc}}$ preserve 
    finite limits. Moreover, since regular epimorphisms in $\CHBR$ are the surjective morphisms by Corollary \ref{lem:regepi}, and this also happens in $\mathsf{Hopf}_{\Bbbk,\mathrm{coc}}$, clearly the functors $F^{\cdot}$ and $F^{\bullet}$ preserve and reflect regular epimorphisms.
\end{proof}

\begin{proposition}
The category $\CHBR$ is regular.
\end{proposition}

\begin{proof}
    By Proposition \ref{prop:factorization} we know that any morphism in $\CHBR$ factorizes as a regular epimorphism followed by a monomorphism. In order to conclude that $\CHBR$ is regular we have to prove that regular epimorphisms are stable under pullbacks. This easily follows from  Lemma \ref{lem:regfunctor} and the fact that regular epimorphisms in $\mathsf{Hopf}_{\Bbbk,\mathrm{coc}}$ are stable under pullbacks (see \cite{GSV}).
\end{proof}

Since we already know that $\CHBR$ is pointed and protomodular (by Proposition \ref{prop:protomodularity}) we obtain:

\begin{theorem}\label{Homological}
The category $\CHBR$ is homological.
\end{theorem}

\section{The category $\CHBR$ is semi-abelian}

It is shown by Agore and Chirvăsitu in \cite{AC} that the category $\CHBR$ has (binary) coproducts. Accordingly, by using the equivalent characterization for semi-abelian categories given in \cite[3.7]{JMT}, we deduce that $\CHBR$ is semi-abelian provided that the direct image of a kernel is again a kernel.

\begin{lemma}\label{lem:imagekernel}
     The direct image of a kernel in $\CHBR$ is a kernel in $\CHBR$.
\end{lemma}

\begin{proof}
    Let $p:(A,\cdot,\bullet)\to(B,\cdot,\bullet)$ be a surjective morphism in $\CHBR$ and $(D,\cdot,\bullet)$ a normal sub-Hopf brace of $(A,\cdot,\bullet)$ in the sense of Definition \ref{def:normalHopfbrace}. We prove that $(p(D),\cdot,\bullet)$ is a normal sub-Hopf brace of $(B,\cdot,\bullet)$. Since $p:A^{\cdot}\to B^{\cdot}$ and $p:A^{\bullet}\to B^{\bullet}$ are surjective morphisms in $\mathsf{Hopf}_{\Bbbk,\mathrm{coc}}$ and $D^{\cdot}$ and $D^{\bullet}$ are normal Hopf subalgebras of $A^{\cdot}$ and $A^{\bullet}$, respectively, we get that $p(D^{\cdot})$ and $p(D^{\bullet})$ are normal Hopf subalgebras of $B^{\cdot}$ and $B^{\bullet}$, respectively, using \cite[Lemma 2.7 (i)]{GSV}. Using $A\rightharpoonup D\subseteq D$, it just remains to prove $B\rightharpoonup p(D)\subseteq p(D)$. Given
    $p(d)\in p(D)$ and
    $b\in B$, which we rewrite as $b=p(a)$, for an $a\in A$, by the surjectivity of $p$, we have
\[
\begin{split}
    b\rightharpoonup p(d)=S(p(a)_{1})\cdot(p(a)_{2}\bullet p(d))=p(S(a_{1})\cdot(a_{2}\bullet d))=p(a\rightharpoonup d)\in p(D)
\end{split}
\]
since $p$ is a morphism in $\CHBR$ and $A\rightharpoonup D\subseteq D$.
\end{proof}

Accordingly, we now get one of the main results of this paper:

\begin{theorem}\label{thm:semi-ab}
    The category $\CHBR$ is semi-abelian.
\end{theorem}

Note that the existence of general coequalizers in $\CHBR$ follows from Theorem \ref{thm:semi-ab}, see e.g. \cite[Proposition 5.1.3]{BB}. We refer the reader to \cite{AC}, where an explicit description of coequalizers in $\CHBR$ was first given. 


\begin{remark}
Observe that the full subcategory of $\CHBR$ whose objects are cocommutative Hopf braces with $\cdot=\bullet$ (the trivial Hopf braces) is closed under finite limits, regular quotients and binary coproducts, hence it is semi-abelian. This category is isomorphic to the category $\mathsf{Hopf}_{\Bbbk,\mathrm{coc}}$, thus we recover \cite[Theorem 2.10]{GSV}.  
\end{remark}

As a consequence of Theorem \ref{thm:semi-ab} we obtain the following result:

\begin{corollary}
    The category $\mathsf{Ab}(\CHBR)$ of abelian objects is abelian.
\end{corollary}

\begin{proof}
The category of abelian objects in an exact category is always abelian.
\end{proof}

In the next subsection we give an explicit description of the category $\mathsf{Ab}(\CHBR)$. 

\subsection{Abelian objects in $\CHBR$} In \cite[Proposition 9]{Bourn} it is proved that an object $C$ in a semi-abelian category $\Cc$ is abelian if and only if its diagonal $\langle\mathrm{Id}_{C},\mathrm{Id}_{C}\rangle:C\to C\times C$ is a kernel in $\Cc$. We can then prove the following result.

\begin{proposition}\label{abelian}
    The category $\mathsf{Ab}(\CHBR)$ is given by those cocommutative Hopf braces $(H,\cdot,\bullet)$ such that $\cdot=\bullet$ is abelian, i.e. it is the category of trivial cocommutative Hopf braces with commutative product.
\end{proposition}

\begin{proof}
The binary product of any $(C,\cdot,\bullet)\in\CHBR$ with itself is given by $(C\otimes C,\cdot_{\otimes},\bullet_{\otimes})$ and $\langle\mathrm{Id}_{C},\mathrm{Id}_{C}\rangle=(\mathrm{Id}_{C}\otimes\mathrm{Id}_{C})\Delta_{C}=\Delta_{C}$. Hence $\mathsf{Ab}(\CHBR)$ is given by those cocommutative Hopf braces $(C,\cdot,\bullet)$ such that $\Delta_{C}$ is a kernel in $\CHBR$. But $\Delta_{C}$ is injective so, writing $\Delta_{C}=i\circ\Delta'_{C}$ where $\Delta'_{C}:(C,\cdot,\bullet)\to(\mathrm{Im}(\Delta),\cdot_{\otimes},\bullet_{\otimes})$ is an isomorphism in $\CHBR$ and $i:(\mathrm{Im}(\Delta),\cdot_{\otimes},\bullet_{\otimes})\hookrightarrow(C\otimes C,\cdot_{\otimes},\bullet_{\otimes})$ is the inclusion, then $\Delta_{C}$ is a kernel in $\CHBR$ if and only if $i$ is a kernel in $\CHBR$. By Proposition \ref{prop:normalmono} we have that $i:(\mathrm{Im}(\Delta),\cdot_{\otimes},\bullet_{\otimes})\hookrightarrow(C\otimes C,\cdot_{\otimes},\bullet_{\otimes})$ is a kernel in $\CHBR$ if and only if $\mathrm{Im}(\Delta)$ is a normal sub-Hopf brace (Definition \ref{def:normalHopfbrace}), i.e. $\mathrm{Im}(\Delta)^{\cdot_{\otimes}}$ and $\mathrm{Im}(\Delta)^{\bullet_{\otimes}}$ are normal Hopf subalgebras of $(C\otimes C)^{\cdot_{\otimes}}$ and $(C\otimes C)^{\bullet_{\otimes}}$, respectively, and $(C\otimes C)\rightharpoonup\mathrm{Im}(\Delta)\subseteq\mathrm{Im}(\Delta)$. Hence $(H,\cdot,\bullet)$ is in $\mathsf{Ab}(\CHBR)$ if and only if $\cdot$ and $\bullet$ are abelian (see \cite[Theorem 2.11]{GSV} and \cite{VW}) and $(C\otimes C)\rightharpoonup\mathrm{Im}(\Delta)\subseteq\mathrm{Im}(\Delta)$. Given $a,b\in C$ and $\Delta(x)=x_{1}\otimes x_{2}\in\mathrm{Im}(\Delta)$, we obtain
\begin{equation}\label{eq:normaldelta}
\begin{split}
    (a\otimes b)\rightharpoonup(x_{1}\otimes x_{2})&=S^{\otimes}((a\otimes b)_{1})\cdot_{\otimes}((a\otimes b)_{2}\bullet_{\otimes}(x_{1}\otimes x_{2}))\\&=(S(a_{1})\otimes S(b_{1}))\cdot_{\otimes}((a_{2}\bullet x_{1})\otimes(b_{2}\bullet x_{2}))\\&=(S(a_{1})\cdot(a_{2}\bullet x_{1}))\otimes(S(b_{1})\cdot(b_{2}\bullet x_{2}))\\&=(a\rightharpoonup x_{1})\otimes(b\rightharpoonup x_{2}).
\end{split}
\end{equation}
Suppose that $(C\otimes C)\rightharpoonup\mathrm{Im}(\Delta)\subseteq\mathrm{Im}(\Delta)$ is satisfied. For any $a,b,x\in C$, there exists $d\in D$ such that $(a\rightharpoonup x_{1})\otimes(b\rightharpoonup x_{2})=d_{1}\otimes d_{2}$. Hence, since $\Delta$ is left linear with respect to $\rightharpoonup$, we obtain $(a\rightharpoonup x_{1})\otimes(b_{1}\rightharpoonup x_{2})\otimes(b_{2}\rightharpoonup x_{3})=d_{1}\otimes d_{2}\otimes d_{3}$ and, by applying $\varepsilon\otimes\mathrm{Id}_{C\otimes C}$ and by using that $\varepsilon$ is left linear with respect to $\rightharpoonup$, we get $\varepsilon(a)\Delta(b\rightharpoonup x)=d_{1}\otimes d_{2}=(a\rightharpoonup x_{1})\otimes(b\rightharpoonup x_{2})$. Hence, by applying $\mathrm{Id}\otimes\varepsilon$, we get $\varepsilon(a)b\rightharpoonup x=\varepsilon(b)a\rightharpoonup x$, for any $a,b,x\in C$. For $a=1$, one gets $b\rightharpoonup x=\varepsilon(b)x$, for any $b,x\in C$, so that $\rightharpoonup$ is the trivial action.
On the other hand, if the action $\rightharpoonup$ is trivial, then
\[
(a\otimes b)\rightharpoonup(x_{1}\otimes x_{2})\overset{\eqref{eq:normaldelta}}{=}(a\rightharpoonup x_{1})\otimes(b\rightharpoonup x_{2})=\varepsilon(ab)x_{1}\otimes x_{2}\in\mathrm{Im}(\Delta), 
\]
and the condition $(C\otimes C)\rightharpoonup\mathrm{Im}(\Delta)\subseteq\mathrm{Im}(\Delta)$ is satisfied. It follows that $$(C\otimes C)\rightharpoonup\mathrm{Im}(\Delta)\subseteq\mathrm{Im}(\Delta)$$ holds if and only if the action $\rightharpoonup$ is trivial. Finally, if $\rightharpoonup$ is trivial then clearly $\cdot=\bullet$ while, if $a\cdot b=a\bullet b=a_{1}\cdot(a_{2}\rightharpoonup b)$, then
\[
\varepsilon(a)b=S(a_{1})\cdot a_{2}\cdot b=S(a_{1})\cdot a_{2}\cdot(a_{3}\rightharpoonup b)=a\rightharpoonup b,
\]
showing that the action $\rightharpoonup$ is trivial.
\end{proof}
It follows that the category $\mathsf{Ab}(\CHBR)$ is isomorphic to the category of commutative and cocommutative Hopf algebras over $\Bbbk$ which is known to be abelian \cite{Takeuchi}, and coincides with the category $\mathsf{Ab}(\mathsf{Hopf}_{\Bbbk,\mathrm{coc}})$ (see \cite[Theorem 2.11]{GSV}).

\subsection{1-cocycles and matched pairs of actions}

From Theorem \ref{thm:semi-ab}, one can deduce the semi-abelianness of two other categories, namely the categories of 1-cocycles and of matched pairs of actions over a cocommutative Hopf algebra. Let us recall the definitions of these categories.

\begin{definition}[{\cite[Definition 1.10]{AGV}}]
Let $H$ and $A$ be Hopf algebras and assume that $A$ is an object in $\mathsf{Mon}(_{H}\mathfrak{M})$. A \textit{bijective 1-cocycle} is a coalgebra isomorphism $\pi:H\to A$ such that
\begin{equation}\label{eq:1cocycle}
    \pi(hk)=\pi(h_{1})(h_{2}\rightharpoonup\pi(k)),\qquad \text{for all}\ h,k\in H,
\end{equation}
where $\rightharpoonup$ denotes the $H$-action on $A$. As it is said in \cite[Remark 1.11]{AGV}, any bijective 1-cocycle $\pi$ satisfies $\pi(1_{H})=1_{A}$. Let $\pi:H\to A$ and $\nu:K\to B$ be bijective 1-cocycles. A morphism between these bijective 1-cocycles is a pair $(f,g)$ of Hopf algebra maps $f:H\to K$, $g:A\to B$ such that
\[
\nu f=g\pi,\quad g(h\rightharpoonup a)=f(h)\rightharpoonup g(a),\ \text{for all}\ a\in A,\ h\in H.
\]
We denote the category of bijective 1-cocycles by $\mathsf{1C}$.
\end{definition}

In \cite[Theorem 1.12]{AGV} it is proved that $\HBR$ and $\mathsf{1C}$ are equivalent categories. More precisely, the functors 
$F:\HBR\to\mathsf{1C}$ and $G:\mathsf{1C}\to\HBR$
establishing the equivalence are defined on objects and morphisms, respectively, by
\[
F((H,\cdot,\bullet)):=(\mathrm{Id}_{H}:H^{\bullet}\to H^{\cdot}), \quad F(f):=(f,f)
\]
and
\[
G((\pi:H\to A^{\cdot})):=(A,\cdot,\bullet), \quad G((f,g)):=g,
\]
where $a\bullet b=\pi(\pi^{-1}(a)\pi^{-1}(b))$, for all $a,b\in A$. This equivalence restricts to $\CHBR$ and $\mathsf{1C}_{\mathrm{coc}}$, the category of bijective 1-cocycles between cocommutative Hopf algebras. Then, from Theorem \ref{thm:semi-ab}, we obtain:

\begin{corollary}
    The category $\mathsf{1C}_{\mathrm{coc}}$ of bijective 1-cocycles between cocommutative Hopf algebras is semi-abelian.
\end{corollary}

\begin{remark}
We observe that one can define the category $\HBR(\Cc)$ of Hopf braces in any braided monoidal category $\Cc$, as well as bijective 1-cocycles $\mathsf{1C}(\Cc)$ in $\Cc$ \cite{FGRR}. These two categories are still equivalent as it is proved in \cite[Theorem 2]{FGRR}. For an arbitrary braided monoidal category $\Cc$, it is not known if the category $\mathsf{Hopf}_{\mathrm{coc}}(\Cc)$ of cocommutative Hopf monoids in $\Cc$ is semi-abelian, hence 
the study of the semi-abelianness of $\CHBR(\Cc)$ seems to be out of reach, for the moment. Some other results in this direction have been obtained in \cite{Sciandra2}, where the semi-abelianness was established in case $\Cc=\mathsf{Vec}_{G}$, the category of $G$-graded vector spaces, if the abelian group $G$ is finitely generated and the characteristic of the base field $\Bbbk$ is different from 2 (not needed if $G$ is finite of odd cardinality), see \cite[Theorem 6.1]{Sciandra2}. This result extends \cite[Theorem 2.10]{GSV}, as the category of cocommutative Hopf algebras is recovered by choosing $G$ equal to the trivial group. Hence it would be reasonable to study, under the same assumptions on $G$ and $\Bbbk$, the exactness properties of the category $\mathsf{HBR}_{\mathrm{coc}}(\mathsf{Vec}_{G})$.
\end{remark}

Under the assumption of cocommutativity, Hopf braces are equivalent to \textit{matched pairs of actions} (in the sense of \cite[Definition 2.1]{FS}). These form a subcategory of the matched pairs of Hopf algebras in the sense of \cite[Definition 7.2.1]{Majid-book}. 

\begin{definition}
    Let $(H,\bullet,1,\Delta,\varepsilon,T)$ be a Hopf algebra. A \textit{matched pair of actions} $(H,\rightharpoonup,\leftharpoonup)$ on $H$ is the datum of a left action $\rightharpoonup$ and a right action $\leftharpoonup$ of $H$ on itself, such that $H$ is in $\mathsf{Comon}(_{H^{\bullet}}\mathfrak{M})$ and in $\mathsf{Comon}(\mathfrak{M}_{H^{\bullet}})$ with the respective actions, and the following conditions hold for all $a,b,c\in H$:
\begin{align}
        \label{matchedpair-3}& a\rightharpoonup(b\bullet c)=(a_{1}\rightharpoonup b_{1})\bullet ((a_{2}\leftharpoonup b_{2})\rightharpoonup c);\\
        \label{matchedpair-4}& (a\bullet b)\leftharpoonup c=(a\leftharpoonup(b_{1}\rightharpoonup c_{1}))\bullet (b_{2}\leftharpoonup c_{2});\\
        &\label{braided-commutativity}a\bullet b = (a_1\rightharpoonup b_1)\bullet (a_2\leftharpoonup b_2).
    \end{align}
A morphism of matched pair of actions $f:(H,\rightharpoonup,\leftharpoonup)\to(K,\rightharpoonup,\leftharpoonup)$ is a morphism of Hopf algebras such that
\[
f(a\rightharpoonup b)=f(a)\rightharpoonup f(b),\quad f(a\leftharpoonup b)=f(a)\leftharpoonup f(b).
\]
The category of matched pairs of actions will be denoted by $\mathsf{MP}$.
\end{definition}

\begin{remark}
We observe that, from \eqref{matchedpair-3} and \eqref{matchedpair-4}, the following conditions follow:
\[
a\rightharpoonup1=\varepsilon(a)1,\qquad 1\leftharpoonup a=\varepsilon(a)1,
\]
using that $\Delta$ and $\varepsilon$ are left linear with respect to the action $\rightharpoonup$ and right linear with respect to the action $\leftharpoonup$. 
\end{remark}
We recall the correspondence between cocommutative Hopf braces and matched pairs of actions over cocommutative Hopf algebras, which we denote by $\mathsf{MP}_{\mathrm{coc}}$.

\begin{theorem}[{\cite[Propositons 3.1, 3.2 and Theorem 3.3]{AGV}}]\label{Thm:Hopfbracecond.eq.} Let $(H,\cdot,\bullet)$ be a cocommutative Hopf brace. Then $(H^{\bullet},\rightharpoonup,\leftharpoonup)$ is a matched pair of actions on $H^{\bullet}$ where $\rightharpoonup$ and $\leftharpoonup$ are given by:
\[
a\rightharpoonup b:=S(a_{1})\cdot(a_{2}\bullet b),\qquad a\leftharpoonup b:=T(a_{1}\rightharpoonup b_{1})\bullet a_{2}\bullet b_{2},\quad \text{for all}\ a,b\in H.
\]
Let $(H,\bullet,1,\Delta,\varepsilon,T)$ be a cocommutative Hopf algebra and $(H,\rightharpoonup,\leftharpoonup)$ a matched pair of actions on $H$. Then, $(H,\cdot,\bullet,1,\Delta,\varepsilon,S,T)$ is a cocommutative Hopf brace defining
\[
a\cdot b=a_{1}\bullet(T(a_{2})\rightharpoonup b),\quad S(a)=a_{1}\rightharpoonup T(a_{2}), \quad \text{for all}\ a,b\in H.
\]
These constructions are inverse to each other and determine an isomorphism of categories: $$\CHBR \cong \mathsf{MP}_{\mathrm{coc}}.$$
\end{theorem}

Hence, from Theorem \ref{thm:semi-ab}, we get the following result:

\begin{corollary}\label{MP}
    The category $\mathsf{MP}_{\mathrm{coc}}$ of matched pairs of actions on cocommutative Hopf algebras is semi-abelian.
\end{corollary}

Observe also that in the arXiv preprint \cite[Theorem 2.13]{LST} it is exhibited an equivalence between the category of cocommutative Hopf braces and the category of cocommutative post-Hopf algebras. Hence also the latter category is semi-abelian.

\begin{remark}
Removing the cocommutativity assumption, it has been proven in \cite[Theorems 3.11, 3.24, 3.25]{FS} that the category $\mathsf{MP}$ of matched pairs of actions on Hopf algebras is isomorphic to the category $\mathsf{YDBr}$ of \textit{Yetter--Drinfeld braces}, introduced in \cite[Definition 3.16]{FS}. Moreover, the latter category is also isomorphic to the category $\mathsf{YDPH}$ of \textit{Yetter--Drinfeld post-Hopf algebras}, introduced in \cite[Definition 3.1]{Sciandra1}, as it is proved in \cite[Proposition 3.21 and Corollaries 3.20, 3.23]{Sciandra1}. It would be interesting to investigate some exactness properties of the latter categories. This goes beyond the scope of this work and we leave it as a possible future project.
\end{remark}

\section{A torsion theory in $\CHBR$}
First recall that in any pointed category $\mathcal C$ one can define a \textit{torsion theory} $({\mathcal T}, {\mathcal F})$ as a pair of full replete subcategories of $\mathcal C$ satifying the following two properties:
\begin{itemize}
\item[1)] for any $X$ in $\mathcal C$ there is a short exact sequence 
$$\begin{tikzcd}
	0 & T & X & F & 0
	\arrow[from=1-1, to=1-2]
	\arrow[from=1-2, to=1-3]
	\arrow[from=1-3, to=1-4]
	\arrow[from=1-4, to=1-5]
\end{tikzcd}$$
where $T \in \mathcal T$ and $F \in \mathcal F$;
\item[2)]  the only morphism from a $T \in \mathcal T$ to an $F\in \mathcal F$ is the zero morphism.
\end{itemize}
When $\mathcal C$ is homological (hence, in particular, when it is semi-abelian) the \textit{torsion-free} subcategory $\mathcal F$ is then (normal epi)-reflective
$$
\begin{tikzcd}
	{\mathcal{F}} & \perp & {\mathcal{C},}
	\arrow["U"', shift right=2, from=1-1, to=1-3]
	\arrow["{\textsf{F}}"', shift right=3, from=1-3, to=1-1]
\end{tikzcd}
$$ in the sense that each component of the unit of the adjunction 
above
is a normal epimorphism. Dually, the \textit{torsion} subcategory $\mathcal T$ is (normal mono)-coreflective in $\mathcal C$,
see e.g. \cite{BG-TT}. A torsion theory is called \textit{hereditary} if $\mathcal{T}$ is closed under subobjects. \smallskip

In the category $\CHBR$ there are two full replete subcategories that will turn out to form a torsion theory.

\begin{itemize}
    \item The category $\mathsf{SKB}$ that is equivalent to the category of \emph{skew braces} \cite{GV}. This subcategory has objects given by the cocommutative Hopf braces whose ``underlying'' Hopf algebras are group Hopf algebras: these are generated as algebras by group-like elements, i.e. by the elements $x$ such that $\Delta(x) =x \otimes x$ and $\varepsilon (x) = 1$. The compatibility \eqref{eq:HBRcomp.} becomes
\[
g\bullet(h\cdot k)=(g\bullet h)\cdot g^{-\cdot}\cdot(g\bullet k),
\]
for all $(\Bbbk G,\cdot,\bullet)$ in $\mathsf{SKB}$ and $g,h,k\in G$.

\medskip
    
\item The category $\mathsf{PHBR}_{\mathrm{coc}}$ of \emph{primitive Hopf braces}. This subcategory has objects given by the cocommutative Hopf braces whose ``underlying'' Hopf algebras are universal enveloping algebras: these are generated as algebras by primitive elements, i.e. by the elements $x$ such that $\Delta(x)=1\otimes x+x\otimes 1$ (from which $\varepsilon(x)=0$ and $S(x)=-x$). The compatibility \eqref{eq:HBRcomp.} becomes 
\[
x\bullet(y\cdot z)=(x\bullet y)\cdot z-y\cdot x\cdot z+y\cdot(x\bullet z)
\]
for all $(H,\cdot,\bullet)$ in $\mathsf{PHBR}_{\mathrm{coc}}$ and $x,y,z\in H$ primitive elements.
\end{itemize}

\begin{remark}
    Since the category $\CHBR$ is pointed and protomodular (see Proposition \ref{prop:protomodularity}), a sequence
\[
\begin{tikzcd}
	0 & T & X & F & 0
	\arrow[from=1-1, to=1-2]
	\arrow[from=1-2, to=1-3, "k"]
	\arrow[from=1-3, to=1-4, "q"]
	\arrow[from=1-4, to=1-5]
\end{tikzcd}
\]
in $\CHBR$ is a short exact sequence precisely when $q$ is a regular epimorphism and $k=\mathsf{ker}(q)$ since, when this is the case, $q = \mathsf{coker}(k)$.   
Moreover, by Corollary \ref{lem:regepi}, we know that regular epimorphisms coincide with surjective morphisms in $\CHBR$. Thus, a short exact sequence in $\CHBR$ has the form
\[
\begin{tikzcd}
	(\mathrm{Hker}(f),\cdot,\bullet) & (A,\cdot,\bullet) & (B,\cdot,\bullet)
	\arrow[hook, from=1-1, to=1-2]
	\arrow[from=1-2, to=1-3, "f"]
\end{tikzcd}
\]
where $f:(A,\cdot,\bullet)\to(B,\cdot,\bullet)$ is a surjective morphism in $\CHBR$.
\end{remark}

\begin{proposition}\label{short-exact}
Let $\Bbbk$ be an algebraically closed field of characteristic 0. For any object $(H,\cdot,\bullet)$ in $\CHBR$ there is a short exact sequence  \begin{equation}\label{Short-HopfB} 
\begin{tikzcd}
	(\mathrm{Hker}(\pi),\cdot,\bullet) & (H,\cdot,\bullet) & (\Bbbk G,\cdot,\bullet)
	\arrow[hook, from=1-1, to=1-2]
	\arrow[from=1-2, to=1-3, "\pi"]
\end{tikzcd}
\end{equation}
such that $(\mathrm{Hker}(\pi),\cdot,\bullet) \in\mathsf{PHBR}_{\mathrm{coc}}$ and $(\Bbbk G,\cdot,\bullet) \in \mathsf{SKB}$, where $G:=G(H)$ is the set of group-like elements of $H$.
\end{proposition}

\begin{proof}
Given an object $(H,\cdot,\bullet)$ in $\CHBR$, we consider the sub-Hopf brace $(\Bbbk G,\cdot,\bullet)\subseteq(H,\cdot,\bullet)$, where $G:=G(H)$ is the set of group-like elements of $H$, which is an object in $\mathsf{SKB}$. In order to complete the proof, we are going to define a surjective morphism $\pi:(H,\cdot,\bullet)\to(\Bbbk G,\cdot,\bullet)$ in $\CHBR$ such that $(\mathrm{Hker}(\pi),\cdot,\bullet)$ is in $\mathsf{PHBR}_{\mathrm{coc}}$. 

Since $\Bbbk$ is algebraically closed, the cocommutative Hopf algebras $H^{\cdot}$ and $H^{\bullet}$ over $\Bbbk$ are automatically pointed (i.e. with coradical given by $\Bbbk G$), see \cite[Lemma 8.0.1 (c)]{Sw}. As coalgebras we have $H=\bigoplus_{g\in G}{H^{g}}$, where $H^{g}$ denotes the irreducible component containing $g\in G$. Defining $\pi_{g}:H^{g}\to\Bbbk G$ by the assignment $h\mapsto\varepsilon(h)g$ and $\pi:=\bigoplus_{g\in G}{\pi_{g}}:H\to\Bbbk G$ one sees that $\pi:H^{\cdot}\to\Bbbk G^{\cdot}$ and $\pi:H^{\bullet}\to\Bbbk G^{\bullet}$ are surjective Hopf algebra maps as proved in \cite[Proposition 8.1.2]{Sw}. Hence $\pi:(H,\cdot,\bullet)\to(\Bbbk G,\cdot,\bullet)$ is a morphism in $\CHBR$. Now we consider the kernel $(\mathrm{Hker}(\pi),\cdot,\bullet)\hookrightarrow(H,\cdot,\bullet)$ of $\pi$ in $\CHBR$. By \cite[Proposition 8.1.4]{Sw} we have that $\mathrm{Hker}(\pi)=H^{1}$, as vector spaces. Observe also that $P(H^{1})=P(H)$, where $P(H)$ denotes the vector space of primitive elements of $H$. 

Since $\Bbbk$ has characteristic 0 and $(H^{1})^{\cdot}$ and $(H^{1})^{\bullet}$ are irreducible cocommutative Hopf algebras, by \cite[Theorem 13.0.1]{Sw} we have that $(H^{1})^{\cdot}$ and $(H^{1})^{\bullet}$ are isomorphic as Hopf algebras to the universal enveloping algebras $U(\mathfrak{g}_{\cdot})$ and $U(\mathfrak{g}_{\bullet})$, respectively, where the vector space $\mathfrak{g}:=P(H)$ is considered as a Lie algebra with respect to two different brackets, namely $[x,y]_{\cdot}=x\cdot y-y\cdot x$ and $[x,y]_{\bullet}=x\bullet y-y\bullet x$, respectively. Hence $(H^{1})^{\cdot}$ and $(H^{1})^{\bullet}$ are generated as algebras by primitive elements and $(\mathrm{Hker}(\pi),\cdot,\bullet)=(H^{1},\cdot,\bullet)$ is an object in $\mathsf{PHBR}_{\mathrm{coc}}$. 
\end{proof}

Under the same assumptions on the base field $\Bbbk$, we can better describe the cocommutative Hopf brace $(\mathrm{Hker}(\pi),\cdot,\bullet)=(H^{1},\cdot,\bullet)$. In particular, we can relate the Hopf algebras $U(\mathfrak{g}_{\cdot})$ and $U(\mathfrak{g}_{\bullet})$, where $\mathfrak{g}:=P(H)=P(H^{1})$. In order to do this, we first recall the following definition:

\begin{definition}[\cite{Vallette}]
    A \textit{post-Lie algebra} $(\mathfrak{g},[\cdot,\cdot],\rightharpoonup)$ is the datum of a Lie algebra $(\mathfrak{g},[\cdot,\cdot])$ and a linear map $\rightharpoonup:\mathfrak{g}\ot\mathfrak{g}\to\mathfrak{g}$ such that, for all $x,y,z\in\mathfrak{g}$, the following equalities hold:
\begin{itemize}
    \item[1)] $x\rightharpoonup[y,z]=[x\rightharpoonup y,z]+[y,x\rightharpoonup z]$; \medskip
    \item[2)] $([x,y]+(x\rightharpoonup y)-(y\rightharpoonup x))\rightharpoonup z=(x\rightharpoonup(y\rightharpoonup z))-(y\rightharpoonup(x\rightharpoonup z))$.
\end{itemize}
\end{definition}

If $(\mathfrak{g},[\cdot,\cdot])$ is abelian, one recovers the famous notion of \textit{pre-Lie algebra} (or \textit{left symmetric algebra}). To any post-Lie algebra $(\mathfrak{g},[\cdot,\cdot],\rightharpoonup)$ one can associate a Lie algebra $\mathfrak{g}_{\rightharpoonup}:=(\mathfrak{g},[\cdot,\cdot]_\rightharpoonup)$, called the
\textit{subadjacent Lie algebra}, where the Lie bracket $[\cdot,\cdot]_{\rightharpoonup}$ is defined by $[x,y]_{\rightharpoonup}:=(x\rightharpoonup y)-(y\rightharpoonup x)+[x,y]$.
We obtain the following result:
\begin{lemma}
    Under the same assumptions as in Proposition \ref{short-exact}, we have that $(\mathfrak{g}:=P(H),[\cdot,\cdot]_{\cdot},\rightharpoonup)$ is a post-Lie algebra and $U(\mathfrak{g}_{\bullet})$ coincides with $U(\mathfrak{g}_{\cdot_{\rightharpoonup}})$.
\end{lemma}
\begin{proof}
Given the cocommutative Hopf brace $(H,\cdot,\bullet)$, by \cite[Theorem 2.13]{LST} we know that $(H,\cdot,\rightharpoonup)$ is a post-Hopf algebra (see \cite[Definition 2.1]{LST2}), where $\rightharpoonup$ is defined as in \eqref{def:leftaction}. Therefore, by \cite[Theorem 2.7]{LST2} we have that $(\mathfrak{g}:=P(H),[\cdot,\cdot]_{\cdot},\rightharpoonup)$ is a post-Lie algebra. Given $x,y\in P(H)$, we compute 
\[
\begin{split}
    [x,y]_{\bullet}&=x\bullet y-y\bullet x=x_{1}\cdot(x_{2}\rightharpoonup y)-y_{1}\cdot(y_{2}\rightharpoonup x)\\&=x\cdot(1\rightharpoonup y)+(x\rightharpoonup y)-y\cdot(1\rightharpoonup x)-(y\rightharpoonup x)\\&=x\cdot y-y\cdot x+(x\rightharpoonup y)-(y\rightharpoonup x)\\&=[x,y]_{\cdot}+(x\rightharpoonup y)-(y\rightharpoonup x),
\end{split}
\]
so $[\cdot,\cdot]_{\bullet}=[\cdot,\cdot]_{\cdot_{\rightharpoonup}}$, i.e. $\mathfrak{g}_{\bullet}$ coincides with the subadjacent Lie algebra $\mathfrak{g}_{\cdot_{\rightharpoonup}}$ of $\mathfrak{g}_{\cdot}$.
\end{proof}

\begin{remark}
    We recall that, given a post-Lie algebra $(\mathfrak{g},[\cdot,\cdot],\rightharpoonup)$, then $(U(\mathfrak{g}),\widetilde{\rightharpoonup})$ is a cocommutative post-Hopf algebra, where the morphism $\widetilde{\rightharpoonup}$ which extends $\rightharpoonup$ is defined as in \cite[Theorem 2.8]{LST2}. Moreover, $U(\mathfrak{g}_{\rightharpoonup})$ is isomorphic to the subadjacent Hopf algebra $U(\mathfrak{g})_{\widetilde{\rightharpoonup}}$ of $U(\mathfrak{g})$. Let us also observe that universal enveloping algebras over pre-Lie algebras $(\mathfrak{g},\rightharpoonup)$ are shown to provide cocommutative Hopf braces in \cite[Proposition 4.6]{AGV}.
\end{remark}

Hence, under the same assumptions on the base field $\Bbbk$, the Hopf algebras $(H^{1})^{\cdot}$ and $(H^{1})^{\bullet}$ are isomorphic to $U(\mathfrak{g}_{\cdot})$ and its subadjacent Hopf algebra, respectively. We have then shown that the primitive Hopf braces in the short exact sequences \eqref{Short-HopfB} are given by universal enveloping algebras of post-Lie algebras. \medskip

As shown in the proof of Proposition \ref{short-exact}, if $\Bbbk$ is algebraically closed, we have a well-defined morphism $\pi:(H,\cdot,\bullet)\to(\Bbbk G,\cdot,\bullet)$ in $\CHBR$, where $G:=G(H)$. Hence we obtain a split epimorphism 
\[
\begin{tikzcd}
	(H,\cdot,\bullet) & (\Bbbk G,\cdot,\bullet)
	\arrow[shift left, from=1-1, to=1-2,"\pi"]
	\arrow[hook',shift left, from=1-2, to=1-1,"\gamma"]
\end{tikzcd},
\]
in $\CHBR$ where $\gamma$ is the inclusion ($\pi \gamma = \mathrm{Id}_{\Bbbk G}$). Hence, by applying Proposition \ref{thm:points}, we get the following:

\begin{corollary}
    Let $\Bbbk$ be an algebraically closed field. Then, any $(H,\cdot,\bullet)$ in $\CHBR$ is isomorphic to $(H^{1}\otimes\Bbbk G,\cdot_{\#},\overline{\bullet})$ and $(H^{1}\otimes\Bbbk G,\overline{\cdot},\bullet_{\#})$ in $\CHBR$, where $G:=G(H)$. 
\end{corollary}

\begin{remark}\label{isophi}
The isomorphisms of Hopf algebras that are used to transport the Hopf brace structure of $(H,\cdot,\bullet)$ are the ones recalled at the beginning of Section \ref{Section-protomodular} 
\[
\phi^{\cdot}:(H^{1}\otimes\Bbbk G)^{\cdot_{\#}}\to H^{\cdot},\ x\otimes g\mapsto x\cdot g,\quad \phi^{\bullet}:(H^{1}\otimes\Bbbk G)^{\bullet_{\#}}\to H^{\bullet},\ x\otimes g\mapsto x\bullet g,
\]
(see also \cite[Theorem 8.1.5]{Sw}), where $(H^{1}\otimes\Bbbk G)^{\cdot_{\#}}$ and $(H^{1}\otimes\Bbbk G)^{\bullet_{\#}}$ have, respectively, products $\cdot_{\#}$ and $\bullet_{\#}$ defined by
\[
(x\otimes g)\cdot_{\#}(x'\otimes g'):=(x\cdot g\cdot x'\cdot g^{-1})\otimes(g\cdot g'), \quad (x\otimes g)\bullet_{\#}(x'\otimes g'):=(x\bullet g\bullet x'\bullet g^{-1})\otimes(g\bullet g'),
\]
for all $x,x'\in H^{1}$ and $g,g'\in G$.
\end{remark}

As we observed in the proof of Proposition \ref{short-exact}, if $\Bbbk$ has characteristic 0, then $(H^{1})^{\cdot}\cong U(\mathfrak{g}_{\cdot})$ and $(H^{1})^{\bullet}\cong U(\mathfrak{g}_{\bullet})$ as Hopf algebras, where the vector space $\mathfrak{g}:=P(H)$ is considered as a Lie algebra with respect to two different brackets. Observe that $U(\mathfrak{g}_{\cdot})$ and $U(\mathfrak{g}_{\bullet})$ are isomorphic as vector spaces (and as coalgebras); we denote them just by $U(\mathfrak{g})$. Using Lemma \ref{lem:transportingHopfbraces}, we know that $(H^{1},\cdot,\bullet)$ is isomorphic to $(U(\mathfrak{g}),-,\overline{\bullet})$ and $(U(\mathfrak{g}),\overline{\cdot},-)$ in $\CHBR$, where $-$ denotes juxtaposition. By applying Proposition \ref{thm:points}, we get the following:

\begin{corollary}
    Let $\Bbbk$ be an algebraically closed field of characteristic 0. Then, any $(H,\cdot,\bullet)$ in $\CHBR$ is isomorphic to $(U(\mathfrak{g})\otimes\Bbbk G,\cdot_{\#},\overline{\bullet})$ and $(U(\mathfrak{g})\otimes\Bbbk G,\overline{\cdot},\bullet_{\#})$ in $\CHBR$, where $\mathfrak{g}:=P(H)$.
\end{corollary}

\begin{remark}
In case $\Bbbk$ is algebraically closed of characteristic 0, $\phi^{\cdot}$ and $\phi^{\bullet}$ as in Remark \ref{isophi} become the well-known Hopf algebras isomorphisms $H^{\cdot}\cong(U(\mathfrak{g})\otimes\Bbbk G)^{\cdot_{\#}}$ and $H^{\bullet}\cong(U(\mathfrak{g})\otimes\Bbbk G)^{\bullet_{\#}}$. In fact, it is known that any cocommutative Hopf algebra $H$ over an algebraically closed field $\Bbbk$ of characteristic 0, is isomorphic to $U(P(H))\#\Bbbk G(H)$; this result is usually associated with the names Cartier--Gabriel--Kostant--Milnor--Moore \cite{MM}.
\end{remark}

From Proposition \ref{short-exact}, we then obtain the following:

\begin{theorem}\label{hereditary}
Let $\Bbbk$ be an algebraically closed field of characteristic 0. The pair $(\mathsf{PHBR}_{\mathrm{coc}}, \mathsf{SKB})$ is a hereditary torsion theory in $\CHBR$.
\end{theorem}
\begin{proof}
The pair $(\mathsf{PHBR}_{\mathrm{coc}},\mathsf{SKB})  $ respects the first condition in the definition of torsion theory by Proposition \ref{short-exact}.
For the second condition, consider a morphism $f \colon (P,\cdot,\bullet) \rightarrow (\Bbbk G,\cdot,\bullet)$ in $\CHBR$ from a primitive Hopf brace $(P,\cdot,\bullet)$ in $\mathsf{PHBR}_{\mathrm{coc}}$ to a Hopf brace $(\Bbbk G,\cdot,\bullet)$ in $\mathsf{SKB}$. The Hopf algebras $P^{\cdot}$ and $P^{\bullet}$ are generated as algebras by primitive elements, which are preserved by any morphism in $\CHBR$. Since there is no non-zero primitive element in the group Hopf algebras $\Bbbk G^{\cdot}$ and $\Bbbk G^{\bullet}$, it follows that $f$ must be zero, so it factors through the zero object $\Bbbk$ in $\CHBR$. This shows that $(\mathsf{PHBR}_{\mathrm{coc}}, \mathsf{SKB})$ is a torsion theory. To see that $\mathsf{PHBR}_{\mathrm{coc}}$ is closed under subobjects in $\CHBR$ consider any monomorphism (equivalently, by Corollary \ref{lem:regepi}, an injective morphism) $m \colon (H,\cdot,\bullet) \rightarrow (P,\cdot,\bullet)$ in $\CHBR$, with $(P,\cdot,\bullet)$ in $\mathsf{PHBR}_{\mathrm{coc}}$. Since the morphism $m:H^{\cdot}\to P^{\cdot}$ 
in $\mathsf{Hopf}_{\Bbbk,\mathrm{coc}}$ preserves group-like elements and there is no non-trivial such element in $P^{\cdot}$, the injectivity of $m$ implies that there is no non-trivial group-like element in $H^{\cdot}$. Since $H^{\cdot}$ and $H^{\bullet}$ are cocommutative Hopf algebras over an algebraically closed field of characteristic 0, through the isomorphisms $H^{\cdot}\cong(U(\mathfrak{g})\otimes\Bbbk G)^{\cdot_{\#}}$ and $H^{\bullet}\cong(U(\mathfrak{g})\otimes\Bbbk G)^{\bullet_{\#}}$, where $\mathfrak{g}:=P(H)$, we deduce that $H^{\cdot}$ and $H^{\bullet}$ are generated as algebras by primitive elements. Hence $(H,\cdot,\bullet)$ belongs to $\mathsf{PHBR}_{\mathrm{coc}}$.
\end{proof}
We observe that the forgetful functors $F^{\cdot}, F^{\bullet}:\CHBR\to\mathsf{Hopf}_{\Bbbk,\mathrm{coc}}$ send the short exact sequence \eqref{Short-HopfB}
to the canonical short exact sequence in the hereditary torsion theory $(\mathsf{PrimHopf}_{\Bbbk}, \mathsf{GrpHopf}_{\Bbbk})$ in $\mathsf{Hopf}_{\Bbbk, \mathrm{coc}}$ \cite[Theorem 4.3]{GKV}. These are so \emph{torsion theory functors} in the sense of \cite{BCG}.

\begin{corollary}\label{Birkhoff}
The category $\mathsf{SKB}$ is a Birkhoff subcategory of $\CHBR$, that is, a full replete reflective subcategory that is closed in $\CHBR$ under subobjects and regular quotients. 
\end{corollary}
\begin{proof}
It is well known that, when $(\mathcal T, \mathcal F)$ is a torsion theory in a homological category $\mathcal C$, the torsion-free subcategory $\mathcal{F}$ is reflective in $\mathcal C$, with the property that each component of the unit of the adjunction is a normal epimorphism. With respect to the torsion theory $(\mathsf{PHBR}_{\mathrm{coc}},\mathsf{SKB}) $, given a cocommutative Hopf brace $H:=(H,\cdot,\bullet)$, the $H$-component of the unit of the adjunction is precisely the morphism $\pi $ in the short exact sequence \eqref{Short-HopfB}. The subcategory $\mathsf{SKB}$ is also closed in $\CHBR$ under regular quotients, since these latter are surjective morphisms of Hopf braces, and they preserve group-like elements.
\end{proof}
Observe that, as a consequence of the fact that $\mathsf{SKB}$ is a Birkhoff subcategory of $\mathsf{HBR}_{\mathrm{coc}}$, one obtains that $\mathsf{SKB}$ is semi-abelian, recovering so the correspondent well-known result for skew braces \cite{BFP}.

Recall that a reflective subcategory $\mathcal F$ of a finitely complete category $\mathcal C$
$$
\begin{tikzcd}
	{\mathcal{F}} & \perp & {\mathcal{C}}
	\arrow["U"', shift right=2, from=1-1, to=1-3]
	\arrow["{\textsf{F}}"', shift right=3, from=1-3, to=1-1]
\end{tikzcd}
$$
is a \emph{localization} of $\mathcal C $ if the left adjoint $\textsf{F} \colon \mathcal C \rightarrow \mathcal F$ preserves finite limits.
It turns out that $\mathsf{SKB}$ is a localization of $\CHBR$:

\begin{theorem}\label{local}
 The category $\mathsf{SKB}$ is a localization of $\CHBR$. 
\end{theorem}
\begin{proof}
The fact that the reflector $\mathsf{F} \colon \CHBR \rightarrow \mathsf{SKB} $ preserves finite products follows from the fact that $\mathsf{SKB}$ is a Birkhoff subcategory (by Corollary \ref{Birkhoff}), by the arguments in the proof of \cite[Lemma 8.2]{Gran-Galois} and the fact that $\CHBR$ is a semi-abelian category. We still need to check that $\mathsf{F}  \colon \CHBR \rightarrow \mathsf{SKB} $ preserves equalizers. For this, consider the equalizer $e:(E,\cdot,\bullet)\to(A,\cdot,\bullet)$ of two morphisms $f,g:(A,\cdot,\bullet)\to(B,\cdot,\bullet)$ in $\CHBR$, and apply the reflector $\mathsf{F} \colon \CHBR \rightarrow \mathsf{SKB}$ getting the following commutative diagram
\[\begin{tikzcd}
	(E,\cdot,\bullet) & (A,\cdot,\bullet) & (B,\cdot,\bullet) \\
	(\Bbbk G_{E},\cdot,\bullet) & (\Bbbk G_{A},\cdot,\bullet) & (\Bbbk G_{B},\cdot,\bullet)
	\arrow[from=1-1, to=1-2, "e"]
	\arrow[shift right, from=1-1, to=2-1, "\pi_{E}"']
	\arrow[shift left, from=1-2, to=1-3, "f"]
	\arrow[shift right, from=1-2, to=1-3, "g"']
	\arrow[shift right, from=1-2, to=2-2, "\pi_{A}"']
	\arrow[shift right, from=1-3, to=2-3, "\pi_{B}"']
	\arrow[shift right, from=2-1, to=1-1, "\gamma_{E}"']
	\arrow[from=2-1, to=2-2, "\mathsf{F}(e)"']
	\arrow[shift right, from=2-2, to=1-2, "\gamma_{A}"']
	\arrow[shift left, from=2-2, to=2-3, "\mathsf{F}(f)"]
	\arrow[shift right, from=2-2, to=2-3, "\mathsf{F}(g)"']
	\arrow[shift right, from=2-3, to=1-3, "\gamma_{B}"']
\end{tikzcd}\]
where the lower row is the image by $\mathsf{F}$ of the upper one in the subcategory $\mathsf{SKB}$, the arrows $\pi_E$, $\pi_A$ and $\pi_B$ are the components of the unit of the adjunction (as $\pi $ in the short exact sequence \eqref{Short-HopfB}), while $\gamma_E$, $\gamma_A$ and $\gamma_B$ are the inclusions. Now, given any arrow $h \colon (H,\cdot,\bullet) \rightarrow (\Bbbk G_A,\cdot,\bullet)$ in $\mathsf{SKB}$ such that $\mathsf{F}(f) h =\mathsf{F}(g) h$, the commutativity of the diagram shows that $f\gamma_A h = g\gamma_A h$, hence there exists a unique $\phi \colon (H,\cdot,\bullet) \rightarrow (E,\cdot,\bullet)$ in $\CHBR$ such that $e \phi = \gamma_A h$. Using again the commutativity of the diagram and the fact that $\pi_{A}\gamma_{A}=\mathrm{Id}_{\Bbbk G_{A}}$, one then checks that the composite $\pi_E \phi \colon (H,\cdot,\bullet) \rightarrow (\Bbbk G_{E},\cdot,\bullet)$ is the desired factorization such that $\mathsf{F}(e) \pi_E \phi = h$. The uniqueness of the factorization follows from the fact that $\mathsf{F}(e)$ is a monomorphism in $\mathsf{SKB}$, since any reflector from a semi-abelian category to a torsion-free subcategory of a hereditary torsion theory preserves monomorphisms \cite{BG-TT}, and  the torsion theory $(\mathsf{PHBR}_{\mathrm{coc}}, \mathsf{SKB})$ is hereditary by Theorem \ref{hereditary}.
\end{proof}

An in-depth study of the category $\mathsf{PHBR}_{\mathrm{coc}}$ is worthwhile but lies beyond the scope of this paper. This will be addressed in future work.

\section{Strong protomodularity of the category $\CHBR$}

The category $\mathsf{Hopf}_{\Bbbk,\mathrm{coc}}$ is not only protomodular but also \textit{strongly protomodular} \cite{Bourn-00}, since it has finite limits and it is the category of internal groups in $\mathsf{Coalg}_{\Bbbk,\mathrm{coc}}$, see e.g. \cite{Street}. We extend this result to the category $\CHBR$.

\begin{theorem}\label{thm:strongproto}
    The category $\CHBR$ is strongly protomodular.
\end{theorem}

\begin{proof}
    Consider the following commutative diagram in $\CHBR$
\[
\begin{tikzcd}
	(\mathrm{Hker}(\pi),\cdot,\bullet) & (A,\cdot,\bullet) & (B,\cdot,\bullet) \\
	(\mathrm{Hker}(\pi'),\cdot,\bullet) & (A',\cdot,\bullet) & (B,\cdot,\bullet)
	\arrow[hook, from=1-1, to=1-2, "j"]
	\arrow[hook, from=1-1, to=2-1, "u"']
	\arrow[shift left, from=1-2, to=1-3, "\pi"]
	\arrow[hook,from=1-2, to=2-2, "f"]
	\arrow[shift left, from=1-3, to=1-2, "\gamma"]
	\arrow[from=1-3, to=2-3, "\mathrm{Id}_{B}"]
	\arrow[hook, from=2-1, to=2-2, "j'"']
	\arrow[shift left, from=2-2, to=2-3, "\pi'"]
	\arrow[shift left, from=2-3, to=2-2, "\gamma'"]
\end{tikzcd}
\]
where $u$ is a kernel in $\CHBR$. In particular, $u:\mathrm{Hker}(\pi)^{\cdot}\to\mathrm{Hker}(\pi')^{\cdot}$ and $u:\mathrm{Hker}(\pi)^{\bullet}\to\mathrm{Hker}(\pi')^{\bullet}$ are kernels in $\mathsf{Hopf}_{\Bbbk,\mathrm{coc}}$ as well as $j':\mathrm{Hker}(\pi')^{\cdot}\hookrightarrow A'^{\cdot}$ and $j':\mathrm{Hker}(\pi')^{\bullet}\hookrightarrow A'^{\bullet}$. Hence, using the strong protomodularity for $\mathsf{Hopf}_{\Bbbk,\mathrm{coc}}$, we get that $j'u:\mathrm{Hker}(\pi)^{\cdot}\hookrightarrow A'^{\cdot}$ and $j'u:\mathrm{Hker}(\pi)^{\bullet}\hookrightarrow A'^{\bullet}$ are kernels in $\mathsf{Hopf}_{\Bbbk,\mathrm{coc}}$. By \cite[Corollary 2.3]{GSV} we have that \eqref{eq:normaldot} and \eqref{eq:normalbullet} are satisfied by $\mathrm{Hker}(\pi)^{\cdot}$ and $\mathrm{Hker}(\pi)^{\bullet}$, respectively. Hence, in order to prove that $(\mathrm{Hker}(\pi),\cdot,\bullet)$ is a normal sub-Hopf brace of $(A',\cdot,\bullet)$ or, equivalently, that $j'u:(\mathrm{Hker}(\pi),\cdot,\bullet)\hookrightarrow(A',\cdot,\bullet)$ is a kernel in $\CHBR$, it just remains to show that \eqref{closureunderharpoon} is satisfied for $\mathrm{Hker}(\pi)\subseteq A'$, i.e. $A'\rightharpoonup\mathrm{Hker}(\pi)\subseteq\mathrm{Hker}(\pi)$ (see Proposition \ref{prop:normalmono}). Take $a\in A'$ and $x\in\mathrm{Hker}(\pi)$. In $\mathsf{Hopf}_{\Bbbk,\mathrm{coc}}$, we have the isomorphism of Hopf algebras $\phi^{\bullet}:(\mathrm{Hker}(\pi')\otimes B)^{\bullet_{\#}}\to A'^{\bullet},\ k\otimes b\mapsto k\bullet\gamma'(b)$, so that 
\[
a=\sum_{i} {k_{i}\bullet\gamma'(b_{i})}=\sum_{i}{k_{i}\bullet f(\gamma(b_{i}))}=\sum_{i}{k_{i}\bullet\gamma(b_{i})}, 
\]
for some $k_{i}\in\mathrm{Hker}(\pi')$ and $b_{i}\in B$. Hence, we get
\[
a\rightharpoonup x=\sum_{i}{(k_{i}\bullet\gamma(b_{i})\rightharpoonup x)}=\sum_{i}{(k_{i}\rightharpoonup(\gamma(b_{i})\rightharpoonup x))}
\]
which is an element in
\[
\big(\mathrm{Hker}(\pi')\rightharpoonup(A\rightharpoonup\mathrm{Hker}(\pi))\big)\subseteq\big(\mathrm{Hker}(\pi')\rightharpoonup\mathrm{Hker}(\pi)\big)\subseteq\mathrm{Hker}(\pi),
\]
since $j$ and $u$ are kernels in $\CHBR$.
\end{proof}

As shown in \cite{BG}, an important consequence of the strong protomodularity is that, in the presence of a zero object, the centrality of two equivalence relations $R$ and $S$
in the sense of Smith \cite{Smith}  is equivalent to the fact that the two corresponding normal monomorphisms $N_R$ and $N_S$ ``commute'' in the sense of Huq \cite{Huq} (see the beginning of the next section). The coincidence of these two conditions is usually referred to as the ``Smith is Huq'' condition. We then have the following:

\begin{corollary}\label{SisH}
    The “Smith is Huq” condition holds in the category $\CHBR$.
\end{corollary}

We'll now give a characterization of Huq commutators and central extensions in the category $\CHBR$.

\section{Central extensions and commutators in $\CHBR$}

In a pointed category $\Cc$ with binary products, one says that two subobjects $i:X\to A$ and $j:Y\to A$ of an object $A$ of $\Cc$ \textit{commute}, in the sense of Huq \cite{Huq}, if there exists an arrow $p$ making the following diagram commute
\begin{equation}\label{diag:comobject}
\begin{tikzcd}
	X & X\times Y & Y \\
	& A
	\arrow[from=1-1, to=1-2,"i_1"]
	\arrow[from=1-1, to=2-2, "i"']
	\arrow[dotted, from=1-2, to=2-2, "p"]
	\arrow[from=1-3, to=1-2,"i_2"']
	\arrow[from=1-3, to=2-2, "j"]
\end{tikzcd}
\end{equation}
where $i_1 = (1_X,0)$ and $i_2 = (0, 1_Y)$ are the canonical morphisms induced by the universal property of the product $X \times Y$.
If $\Cc$ is protomodular, such an arrow $p$ is unique, whenever it exists, see \cite{BG}. We analyze the diagram \eqref{diag:comobject} in $\CHBR$. 
\begin{remark}\label{rmk:commutingobject}
Since $\CHBR$ is protomodular (see Proposition \ref{prop:protomodularity}) if such a $p$ exists it must be unique. Moreover, we know that the binary product of $(X,\cdot,\bullet)$ and $(Y,\cdot,\bullet)$ in $\CHBR$ is given by $(X\otimes Y,\cdot_{\otimes},\bullet_{\otimes})$ and 
\[
i_{1}:(X,\cdot,\bullet)\to (X\otimes Y,\cdot_{\otimes},\bullet_{\otimes}),x\mapsto x\otimes1,\quad i_{2}:(Y,\cdot,\bullet)\to (X\otimes Y,\cdot_{\otimes},\bullet_{\otimes}),y\mapsto1\otimes y
\]
are the canonical morphisms induced by the product.
\end{remark}
We obtain the following result.

\begin{lemma}\label{lem:commutator}
    Let $(A,\cdot,\bullet)$ be a cocommutative Hopf brace and $(X,\cdot,\bullet)$, $(Y,\cdot,\bullet)$ two sub-Hopf braces of $(A,\cdot,\bullet)$. The following conditions are equivalent:
\begin{itemize}
    \item[1)] there exists a unique morphism $p:(X\otimes Y,\cdot_{\otimes},\bullet_{\otimes})\to(A,\cdot,\bullet)$ in $\CHBR$ such that the diagram \eqref{diag:comobject} commutes; \medskip
    \item[2)] $y\cdot x=x\cdot y=x\bullet y=y\bullet x$, for all $x\in X$ and $y\in Y$;\medskip
    \item[3)] $x_{1}\cdot y_{1}\cdot S(x_{2})\cdot S(y_{2})=S(x_{1})\cdot(x_{2}\bullet y_{1})\cdot S(y_{2})=x_{1}\bullet y_{1}\bullet T(x_{2})\bullet T(y_{2})=\varepsilon(x)\varepsilon(y)1$, for all $x\in X$ and $y\in Y$.
\end{itemize}
\end{lemma}

\begin{proof}
$1)\iff2)$. By Remark \ref{rmk:commutingobject} we know that the morphism $p:(X\otimes Y,\cdot_{\otimes},\bullet_{\otimes})\to(A,\cdot,\bullet)$ in $\CHBR$ has to satisfy the properties $p(x\otimes1)=x$ for all $x\in X$ and $p(1\otimes y)=y$ for all $y\in Y$. Hence, since $p$ must be a morphism of algebras with respect to both the two algebra structures, we obtain
\[
p(x\otimes y)=p((x\otimes1)\cdot_{\otimes}(1\otimes y))=p(x\otimes1)\cdot p(1\otimes y)=x\cdot y,
\]
for all $x\in X$ and $y\in Y$ and, analogously, $p(x\otimes y)=x\bullet y$ for all $x\in X$ and $y\in Y$. Then, $x\bullet y=x\cdot y$ for all $x\in X$ and $y\in Y$. Moreover, we have
\[
x\cdot y=p(x\otimes y)=p((1\otimes y)\cdot_{\otimes}(x\otimes 1))=p(1\otimes y)\cdot p(x\otimes 1)=y\cdot x
\]
for all $x\in X$ and $y\in Y$. Analogously, we obtain $x\bullet y=y\bullet x$ for all $x\in X$ and $y\in Y$. Clearly, if 2) is satisfied, $p=m_{\cdot}|_{X\otimes Y}=m_{\bullet}|_{X\otimes Y}$ is a morphism in $\CHBR$ such that the diagram \eqref{diag:comobject} commutes. \medskip

\noindent $2)\iff3)$. Since $X^{\cdot}$ and $Y^{\cdot}$ are Hopf subalgebras of $A^{\cdot}$, and $X^{\bullet}$ and $Y^{\bullet}$ are Hopf subalgebras of $A^{\bullet}$ in $\mathsf{Hopf}_{\Bbbk,\mathrm{coc}}$ we get that the equations $y\cdot x=x\cdot y$ and $x\bullet y=y\bullet x$ for all $x\in X$, $y\in Y$ are, respectively, equivalent to the equations $x_{1}\cdot y_{1}\cdot S(x_{2})\cdot S(y_{2})=\varepsilon(x)\varepsilon(y)1$ and $x_{1}\bullet y_{1}\bullet T(x_{2})\bullet T(y_{2})=\varepsilon(x)\varepsilon(y)1$, for all $x\in X$, $y\in Y$, using \cite[Lemma 4.1, $(b)\iff (c)$]{GSV}. It just remains to show that $x\cdot y=x\bullet y$ for all $x\in X$, $y\in Y$ is equivalent to $S(x_{1})\cdot(x_{2}\bullet y_{1})\cdot S(y_{2})=\varepsilon(x)\varepsilon(y)1$ for all $x\in X$, $y\in Y$. Clearly, if we assume that $x\cdot y=x\bullet y$ for all $x\in X$, $y\in Y$ then $$S(x_{1})\cdot(x_{2}\bullet y_{1})\cdot S(y_{2})=S(x_{1})\cdot x_{2}\cdot y_{1}\cdot S(y_{2})=\varepsilon(x)\varepsilon(y)1.$$ On the other hand, we have
\[
\begin{split}
    x\bullet y&=\varepsilon(x_{1})(x_{2}\bullet y_{1})\varepsilon(y_{2})=x_{1}\cdot S(x_{2})\cdot(x_{3}\bullet y_{1})\cdot S(y_{2})\cdot y_{3}\\&\overset{(!)}{=}x_{1}\cdot\varepsilon(x_{2})\varepsilon(y_{1})1\cdot y_{2}=x\cdot y,
\end{split}
\]
where $(!)$ follows assuming $S(x_{1})\cdot(x_{2}\bullet y_{1})\cdot S(y_{2})=\varepsilon(x)\varepsilon(y)1$ for all $x\in X$ and $y\in Y$.
\end{proof}

\begin{remark}
Since $x\bullet y=x_{1}\cdot(x_{2}\rightharpoonup y)$, from $x\bullet y=x\cdot y$ we obtain
\[
\varepsilon(x)y=S(x_{1})\cdot x_{2}\cdot y=S(x_{1})\cdot x_{2}\cdot(x_{3}\rightharpoonup y)=x\rightharpoonup y,
\]
i.e. the action $\rightharpoonup$ of $X$ on $Y$ is trivial. Analogously, from $y\cdot x=y\bullet x$, we get that the action $\rightharpoonup$ of $Y$ on $X$ is trivial.
\end{remark}

\noindent\textbf{Central extensions}. In any semi-abelian category $\mathcal C$ the subcategory $\mathsf{Ab}(\mathcal C)$ of abelian objects in $\mathcal C$ (=internal abelian groups in $\mathcal C$) is a Birkhoff subcategory of $\mathcal C$
$$
\begin{tikzcd}
	{\mathsf{Ab}(\mathcal C)} & \perp & {\mathcal{C}.}
	\arrow["U"', shift right=2, from=1-1, to=1-3]
	\arrow["{\textsf{ab}}"', shift right=3, from=1-3, to=1-1]
\end{tikzcd}
$$ This adjunction induces an \textit{admissible Galois structure} that has been investigated, in the semi-abelian context, in \cite{BG2}. There is a natural notion of central extension in $\mathcal C$ relative to this Galois structure, that turns out to be equivalent to the notion considered in commutator theory, that can be described as a regular epimorphism $f \colon A \rightarrow B$ such that its kernel $\mathsf{Ker}(f) \rightarrow A$ ``commutes'' in the sense of Huq with the largest subobject $1_A \colon A \rightarrow A$ of $A$. In the special case of the semi-abelian category $\CHBR$, where regular epimorphisms coincide with surjective morphisms by Corollary \ref{lem:regepi}, this means that, for a given surjective morphism $f:(A,\cdot,\bullet)\to(B,\cdot,\bullet)$, there is a (unique) morphism $p:(\mathrm{Hker}(f)\otimes A,\cdot_{\otimes},\bullet_{\otimes})\to (A,\cdot,\bullet)$ in $\CHBR$ making the following diagram commute:
$$
\begin{tikzcd}
	(\mathrm{Hker}(f),\cdot,\bullet) & (\mathrm{Hker}(f) \otimes A,\cdot_{\otimes},\bullet_{\otimes}) & (A,\cdot,\bullet). \\
	& (A,\cdot,\bullet)
	\arrow[from=1-1, to=1-2,"i_1"]
	\arrow[hook, from=1-1, to=2-2]
	\arrow[dotted, from=1-2, to=2-2, "p"]
	\arrow[from=1-3, to=1-2,"i_2"']
	\arrow[from=1-3, to=2-2, "1_A"]
\end{tikzcd}
$$
 By using Lemma \ref{lem:commutator} one sees that this condition on the kernel of $f$ is equivalently expressed by the simple properties:
$$ a \cdot k = k \cdot a = k \bullet a = a \bullet k , \qquad \text{for all}\ k \in \mathrm{Hker}(f),\ a \in A.$$
One can then introduce the following notion:
\begin{definition}
A surjective morphism $f:(A,\cdot,\bullet)\to(B,\cdot,\bullet)$ in $\CHBR$ is a \emph{central extension} if the following identities 
\[
a \cdot k = k \cdot a = k \bullet a =a \bullet k
\]
hold for all $k \in \mathrm{Hker}(f)$ and $a \in A$.
\end{definition} 
As it follows from the results in \cite{BG2}, the category of central extensions in $\CHBR$ is then equivalent to the category of (internal) connected groupoids in $\CHBR$.
One can also define the so-called \emph{Baer sum} of two central extensions, as explained at the beginning of \cite[Section 6]{GVdL}, for instance, that equips the set of (isomorphism classes of) central extensions with an abelian group structure. Observe that a notable simplification of these developments, that we leave for future work, is due to the validity of the ``Smith is Huq'' condition in $\CHBR$ (Corollary \ref{SisH}). This has already been observed in the special case of skew braces in \cite{Bourn-Skew}.  \medskip

\noindent\textbf{Huq commutator}.
In a semi-abelian category $\Cc$, the \textit{Huq commutator} of two normal subobjects $i:X\to A$ and $j:Y\to A$ in $\Cc$ is the smallest normal subobject $K\to A$ such that its cokernel $q:A\to A/K$ has the property
that the images $q(X)$ and $q(Y)$ by $q$ commute in the quotient, in the sense of Huq. \medskip

Accordingly, in the category $\CHBR$, the Huq commutator of two normal sub-Hopf braces $(X,\cdot,\bullet)$ and $(Y,\cdot,\bullet)$ of a cocommutative Hopf brace $(A,\cdot,\bullet)$ is the smallest normal sub-Hopf brace $([X,Y]_{\mathrm{Huq}},\cdot,\bullet)\hookrightarrow(A,\cdot,\bullet)$ such that, given its cokernel $\pi:(A,\cdot,\bullet)\to\Big(\frac{A}{A\cdot[X,Y]^{+}_{\mathrm{Huq}}}=\frac{A}{A\bullet[X,Y]^{+}_{\mathrm{Huq}}},\overline{\cdot},\overline{\bullet}\Big)$ in $\CHBR$ (see Proposition \ref{prop:normalmono}), the normal (see Lemma \ref{lem:imagekernel}) sub-Hopf braces $(\pi(X),\overline{\cdot},\overline{\bullet})$ and $(\pi(Y),\overline{\cdot},\overline{\bullet})$ commute in the sense of Huq:
\footnotesize
\[\begin{tikzcd}
	(X,\cdot,\bullet) & (Y,\cdot,\bullet) & (\pi(X),\overline{\cdot},\overline{\bullet}) & (\pi(X)\otimes\pi(Y),\overline{\cdot}_{\otimes},\overline{\bullet}_{\otimes}) & (\pi(Y),\overline{\cdot},\overline{\bullet}) \\
	([X,Y]_{\mathrm{Huq}},\cdot,\bullet) & (A,\cdot,\bullet) && \Big(\frac{A}{A\cdot[X,Y]^{+}_{\mathrm{Huq}}},\overline{\cdot},\overline{\bullet}\Big)
	\arrow[hook, from=1-1, to=2-2]
	\arrow[hook', from=1-2, to=2-2]
	\arrow[from=1-3, to=1-4, "i_{1}"]
	\arrow[hook, from=1-3, to=2-4]
	\arrow[dashed, from=1-4, to=2-4, "p"]
	\arrow[from=1-5, to=1-4, "i_{2}"']
	\arrow[hook', from=1-5, to=2-4]
	\arrow[hook, from=2-1, to=2-2]
	\arrow[from=2-2, to=2-4, "\pi"']
\end{tikzcd}\]
\normalsize
By Lemma \ref{lem:commutator}, the existence of $p$ as in the above diagram in $\CHBR$ is equivalent to the equalities
\[
y\cdot x+A\cdot[X,Y]^{+}_{\mathrm{Huq}}=x\cdot y+A\cdot[X,Y]^{+}_{\mathrm{Huq}}=x\bullet y+A\cdot[X,Y]^{+}_{\mathrm{Huq}}=y\bullet x+A\cdot[X,Y]^{+}_{\mathrm{Huq}}, 
\]
for all $x\in X$ and $y\in Y$ and then to the conditions
\begin{equation}\label{cond:comm}
x\cdot y-y\cdot x\in A\cdot[X,Y]^{+}_{\mathrm{Huq}},\quad x\cdot y-x\bullet y\in A\cdot[X,Y]^{+}_{\mathrm{Huq}},\quad x\bullet y-y\bullet x\in A\bullet[X,Y]^{+}_{\mathrm{Huq}}.
\end{equation}
We shall now give an explicit description of the Huq commutator of two normal sub-Hopf braces $(X,\cdot,\bullet)$ and $(Y,\cdot,\bullet)$ of a cocommutative Hopf brace $(A,\cdot,\bullet)$. We first give the following definition.

\begin{definition}
    Let $(H,\cdot,\bullet)$ be a Hopf brace and $X$ a set. The sub-Hopf brace \textit{generated} by $X$, denoted by $(\langle X\rangle,\cdot,\bullet)$, is the intersection of all the sub-Hopf braces of $(H,\cdot,\bullet)$ containing $X$. The normal sub-Hopf brace of $(H,\cdot,\bullet)$ generated by $X$, denoted by $(\langle X\rangle_{N},\cdot,\bullet)$, is the intersection of all the normal sub-Hopf braces of $(H,\cdot,\bullet)$ containing $X$.
\end{definition}

\begin{remark}
    The intersection of Hopf subalgebras of a given Hopf algebra is clearly a Hopf subalgebra, hence the intersection of sub-Hopf braces of a given Hopf brace is clearly a sub-Hopf brace. 
    So $(\langle X\rangle,\cdot,\bullet)$ is the smallest sub-Hopf brace of $(H,\cdot,\bullet)$ containing $X$. Moreover, the intersection of normal Hopf subalgebras is a normal Hopf subalgebra, hence the intersection of normal sub-Hopf braces is also a normal sub-Hopf brace. Thus $(\langle X\rangle_{N},\cdot,\bullet)$ is the smallest normal sub-Hopf brace of $(H,\cdot,\bullet)$ containing $X$.
\end{remark}

We define the normal sub-Hopf brace $([X,Y],\cdot,\bullet)$ of $(A,\cdot,\bullet)$ where:
\[
[X,Y]:=\langle\{x_{1}\cdot y_{1}\cdot S(x_{2})\cdot S(y_{2}), S(x_{1})\cdot(x_{2}\bullet y_{1})\cdot S(y_{2}), x_{1}\bullet y_{1}\bullet T(x_{2})\bullet T(y_{2})\ |\ x\in X,y\in Y\}\rangle_{N}.
\]
We obtain the following result.

\begin{proposition}\label{HCommutator}
    Let $(X,\cdot,\bullet)$ and $(Y,\cdot,\bullet)$ be two normal sub-Hopf braces of a cocommutative Hopf brace $(A,\cdot,\bullet)$. Then, $[X,Y]=[X,Y]_{\mathrm{Huq}}$.
\end{proposition}

\begin{proof}
By definition we know that $([X,Y],\cdot,\bullet)$ is the smallest normal sub-Hopf brace of $(A,\cdot,\bullet)$ containing $X$. We consider its cokernel in $\CHBR$, which is given by $\pi:(A,\cdot,\bullet)\to\Big(\frac{A}{A\cdot[X,Y]^{+}}=\frac{A}{A\bullet[X,Y]^{+}},\overline{\cdot},\overline{\bullet}\Big)$, see Proposition \ref{prop:normalmono}. First we show that the normal (see Lemma \ref{lem:imagekernel}) sub-Hopf braces $(\pi(X),\overline{\cdot},\overline{\bullet})$ and $(\pi(Y),\overline{\cdot},\overline{\bullet})$ commute in the sense of Huq. As we already observed, in $\CHBR$, this is equivalent to the three conditions \eqref{cond:comm}. As in the proof of \cite[Proposition 4.3]{GSV} one has
\[
x\cdot y-y\cdot x=x_{1}\cdot y_{1}\cdot\big(\varepsilon(x_{2})\varepsilon(y_{2})1-S(y_{2})\cdot S(x_{2})\cdot y_{3}\cdot x_{3}\big)\in A\cdot[X,Y]^{+},
\]
by definition of $[X,Y]^{+}$ and, analogously, $x\bullet y-y\bullet x\in A\bullet[X,Y]^{+}$. Moreover, we have
\[
\begin{split}
    x\cdot y-x\bullet y=x_{1}\cdot\big(\varepsilon(x_{2})\varepsilon(y_{1})\varepsilon(y_{2})1-S(x_{2})\cdot(x_{3}\bullet y_{1})\cdot S(y_{2})\big)\cdot y_{3}\in A\cdot[X,Y]^{+},
\end{split}
\]
by definition of $[X,Y]^{+}$. Finally, we show that $([X,Y],\cdot,\bullet)$ is the smallest normal sub-Hopf brace satisfying \eqref{cond:comm}, i.e. that it factorizes through the categorical kernel of any surjective morphism of Hopf braces for which the images of $(X,\cdot,\bullet)$ and $(Y,\cdot,\bullet)$ commute. Let $f:(A,\cdot,\bullet)\to(B,\cdot,\bullet)$ be any surjective morphism in $\CHBR$ such that $(f(X),\cdot,\bullet)\hookrightarrow(B,\cdot,\bullet)$ and $(f(Y),\cdot,\bullet)\hookrightarrow(B,\cdot,\bullet)$ commute. As in the proof of \cite[Proposition 4.3]{GSV} one can show that $x_{1}\cdot y_{1}\cdot S(x_{2})\cdot S(y_{2})\in\mathrm{Hker}(f)$ and, analogously, $x_{1}\bullet y_{1}\bullet T(x_{2})\bullet T(y_{2})\in\mathrm{Hker}(f)$, for all $x\in X$ and $y\in Y$. It only remains to prove that $S(x_{1})\cdot(x_{2}\bullet y_{1})\cdot S(y_{2})\in\mathrm{Hker}(f)$, for all $x\in X$ and $y\in Y$. Using the cocommutativity we have:
\[
\begin{split}
    \big(S(x_{1})\cdot(x_{2}\bullet y_{1})\cdot &S(y_{2})\big)\otimes f(S(x_{3})\cdot(x_{4}\bullet y_{3})\cdot S(y_{4}))\\&= \big(S(x_{1})\cdot(x_{2}\bullet y_{1})\cdot S(y_{2})\big)\otimes\big( f(S(x_{3}))\cdot(f(x_{4})\bullet f(y_{3}))\cdot f(S(y_{4}))\big)\\&\overset{(!)}{=}\big(S(x_{1})\cdot(x_{2}\bullet y_{1})\cdot S(y_{2})\big)\otimes\big( f(S(x_{3}))\cdot f(x_{4})\cdot f(y_{3})\cdot f(S(y_{4}))\big)\\&=\big(S(x_{1})\cdot(x_{2}\bullet y_{1})\cdot S(y_{2})\big)\otimes1,
\end{split}
\]
where $(!)$ follows using that $f(x)\cdot f(y)=f(x)\bullet f(y)$ for all $x\in X$ and $y\in Y$. Hence the Huq commutator of $(X,\cdot,\bullet)$ and $(Y,\cdot,\bullet)$ is $([X,Y],\cdot,\bullet)$.
\end{proof}

\begin{corollary}
The category $\mathsf{Ab} (\CHBR)$ is a Birkhoff subcategory of $\CHBR$.
\end{corollary}
\begin{proof}
It is well known that the abelian objects in a semi-abelian category form a Birkhoff subcategory, and that the reflector of an object $A$ is given by the quotient of $A$ by the largest Huq commutator on $A$. Then, the reflector $\mathsf{ab} \colon \CHBR \rightarrow \mathsf{Ab} (\CHBR)$ sends a cocommutative Hopf brace $(A,\cdot,\bullet)$ to the quotient Hopf brace $(\frac{A}{A\cdot[A,A]_{\mathrm{Huq}}^+},\overline{\cdot},\overline{\bullet})$, where $[A,A]_{\mathrm{Huq}}$ is described as in Proposition \ref{HCommutator}.
\end{proof}

\begin{remark}
The normal sub-Hopf brace $(A*A,\cdot,\bullet)$ of a cocommutative Hopf brace $(A,\cdot,\bullet)$, where 
\[
A*A:=\langle\{ S(x_{1})\cdot(x_{2}\bullet y_{1})\cdot S(y_{2})\ |\ x,y\in A\}\rangle_{N}
\]
allows one to define the left adjoint $F:\CHBR\to\mathsf{Hopf}_{\Bbbk,\mathrm{coc}}$ to the functor $G:\mathsf{Hopf}_{\Bbbk,\mathrm{coc}}\to\CHBR$ which sends a cocommutative Hopf algebra $A^{\cdot}$ into the trivial cocommutative Hopf brace $(A,\cdot,\cdot)$. Indeed, by taking the quotient Hopf brace $F((A,\cdot,\bullet)) :=(\frac{A}{A\cdot(A*A)^{+}},\overline{\cdot},\overline{\bullet})$ one gets the universal trivial cocommutative Hopf brace associated with $(A,\cdot,\bullet)$. Hence, the subcategory of trivial Hopf braces, which is indeed isomorphic to $\mathsf{Hopf}_{\Bbbk, \mathrm{coc}}$, is a Birkhoff subcategory of $\CHBR$.
\end{remark}

In a semi-abelian category, the Huq commutator is known to satisfy several nice properties, as for instance the symmetry $[X,Y]_{\mathsf{Huq}}= [Y,X]_{\mathsf{Huq}},$
or the preservation under direct images: when $f \colon A \rightarrow B$ is a regular epimorphism, then $$f[X,Y]_{\mathsf{Huq}}= [f(X),f(Y)]_{\mathsf{Huq}}.$$
These and other properties of the Huq commutator \cite{Huq, EVdL} could be useful to investigate nilpotent and solvable Hopf braces. These natural developments will be considered in a subsequent work.
\medskip

\noindent\textbf{Acknowledgments}. The authors would like to thank A.L. Agore and A. Chirvăsitu very much for communicating the  explicit construction of binary coproducts in $\CHBR$, whose existence is needed to conclude that $\CHBR$ is semi-abelian. This research was supported by the Fonds de la Recherche Scientifique – FNRS under Grant CDR No.
J.0080.23. This paper was written while Andrea Sciandra was member of the “National Group for Algebraic and Geometric Structures and their Applications” (GNSAGA-INdAM). He was also partially supported by the project funded by the European Union -NextGenerationEU under NRRP, Mission 4 Component 2 CUP D53D23005960006 - Call PRIN 2022 No. 104 of February 2, 2022 of Italian Ministry of University and Research; Project 2022S97PMY Structures for Quivers, Algebras and Representations (SQUARE).

\end{document}